\documentclass[12 pt,oneside,reqno]{article}
\usepackage[english]{babel}

\usepackage[letterpaper,top=2cm,bottom=2cm,left=3cm,right=3cm,marginparwidth=1.75cm]{geometry}

\usepackage[dvips]{graphicx}
\usepackage{calc}
\usepackage{color}
\usepackage{amsmath}
\usepackage{amssymb}
\usepackage{amscd}
\usepackage{amsthm}
\usepackage{amsbsy}
\usepackage{delarray}
\usepackage{tikz}
\usepackage{tikz-cd}
\usepackage{enumerate}
\usepackage{enumitem}
\usepackage[T1]{fontenc}
\usepackage{enumerate}
\usepackage{stackengine} 
\usepackage{mathrsfs}
\usepackage{hyperref}
\usepackage[whole]{bxcjkjatype}
\usepackage{titlesec}
\usepackage[nottoc]{tocbibind}

\usepackage{lipsum}

\definecolor{YK}{rgb}{9,0,0}
\definecolor{YKb}{rgb}{0,0,9}

 \date{\vspace{-4ex}}

\begin{document}



\setlength{\parindent}{5mm}
\renewcommand{\leq}{\leqslant}
\renewcommand{\geq}{\geqslant}
\newcommand{\N}{\mathbb{N}}
\newcommand{\sph}{\mathbb{S}}
\newcommand{\Z}{\mathbb{Z}}
\newcommand{\R}{\mathbb{R}}
\newcommand{\Q}{\mathbb{Q}}
\newcommand{\C}{\mathbb{C}}
\newcommand{\curlL}{\mathcal{L}}
\newcommand{\A}{\mathcal{A}}
\newcommand{\curlP}{\mathcal{P}}
\newcommand{\F}{\mathbb{F}}
\newcommand{\g}{\mathfrak{g}}
\newcommand{\h}{\mathfrak{h}}
\newcommand{\K}{\mathbb{K}}
\newcommand{\RN}{\mathbb{R}^{2n}}
\newcommand{\ci}{c^{\infty}}
\newcommand{\derive}[2]{\frac{\partial{#1}}{\partial{#2}}}
\renewcommand{\S}{\mathbb{S}}
\renewcommand{\H}{\mathbb{H}}
\newcommand{\eps}{\varepsilon}
\newcommand{\lag}{Lagrangian}
\newcommand{\sub}{submanifold}
\newcommand{\homo}{homogeneous}
\newcommand{\qmor}{quasimorphism}
\newcommand{\enum}{enumerate}
\newcommand{\sa}{symplectically aspherical}
\newcommand{\ovl}{\overline}
\newcommand{\wt}{\widetilde}
\newcommand{\hamd}{Hamiltonian diffeomorphism}
\newcommand{\QH}{quantum cohomology ring}
\newcommand{\thm}{Theorem}
\newcommand{\cor}{Corollary}
\newcommand{\hamil}{Hamiltonian}
\newcommand{\propo}{Proposition}
\newcommand{\conjec}{Conjecture}
\newcommand{\asympt}{asymptotic}
\newcommand{\PR}{pseudo-rotation}
\newcommand{\fuk}{\mathscr{F}}
\newcommand{\specinv}{spectral invariant}
\newcommand{\wrt}{with respect to}
\newcommand{\suphv}{superheavy}
\newcommand{\suphvness}{superheaviness}
\newcommand{\symp}{symplectic}
\newcommand\oast{\stackMath\mathbin{\stackinset{c}{0ex}{c}{0ex}{\ast}{\bigcirc}}}
\newcommand{\diffeo}{diffeomorphism}
\newcommand{\quasi}{quasimorphism}
\newcommand{\univham}{\widetilde{\Ham}}
\newcommand{\BK}{Biran--Khanevsky}
\newcommand{\BC}{Biran--Cornea}
\newcommand{\TV}{Tonkonog--Varolgunes}
\newcommand{\varol}{Varolgunes}
\newcommand{\mfd}{manifold}
\newcommand{\smfd}{submanifold}
\newcommand{\EP}{Entov--Polterovich}
\newcommand{\pol}{Polterovich}
\newcommand{\suppot}{superpotential}
\newcommand{\CY}{Calabi--Yau}
\newcommand{\degen}{degenerate}
\newcommand{\coeff}{coefficient}
\newcommand{\acs}{almost complex structure}
\newcommand{\prl}{pearl trajectory}
\newcommand{\prls}{pearl trajectories}
\newcommand{\holo}{holomorphic}
\newcommand{\trans}{transversality}
\newcommand{\pert}{perturbation}
\newcommand{\idem}{idempotent}
\newcommand{\critpt}{critical point}
\newcommand{\tordeg}{toric degeneration}
\newcommand{\nbhd}{neighborhood}
\newcommand{\polar}{polarization}
\newcommand{\symplecto}{symplectomorphism}
\newcommand{\FOOO}{Fukaya--Oh--Ohta--Ono}
\newcommand{\quadr}{quadric hypersurface}
\newcommand{\monoconst}{monotonicity constant}
\newcommand{\constr}{construction}
\newcommand{\coord}{coordinate}
\newcommand{\decomp}{decomposition}
\newcommand{\hypsurf}{hypersurface}
\newcommand{\NNU}{Nishinou--Nohara--Ueda}
\newcommand{\degtion}{degeneration}
\newcommand{\config}{configuration}
\newcommand{\sing}{singular}
\newcommand{\lagsph}{Lagrangian sphere}

\theoremstyle{plain}
\newtheorem{theo}{Theorem}
\newtheorem{theox}{Theorem}
\renewcommand{\thetheox}{\Alph{theox}}
\numberwithin{theo}{subsection}
\newtheorem{prop}[theo]{Proposition}
\newtheorem{lemma}[theo]{Lemma}
\newtheorem{definition}[theo]{Definition}
\newtheorem*{notation*}{Notation}
\newtheorem*{notations*}{Notations}
\newtheorem{corol}[theo]{Corollary}
\newtheorem{conj}[theo]{Conjecture}
\newtheorem{guess}[theo]{Guess}
\newtheorem{claim}[theo]{Claim}
\newtheorem{question}[theo]{Question}
\newtheorem{prob}[theo]{Problem}
\numberwithin{equation}{subsection}

\newenvironment{demo}[1][]{\addvspace{8mm} \emph{Proof #1.
    ~~}}{~~~$\Box$\bigskip}

\newlength{\espaceavantspecialthm}
\newlength{\espaceapresspecialthm}
\setlength{\espaceavantspecialthm}{\topsep} \setlength{\espaceapresspecialthm}{\topsep}

\newenvironment{example}[1][]{\refstepcounter{theo} 
\vskip \espaceavantspecialthm \noindent \textsc{Example~\thetheo
#1.} }%
{\vskip \espaceapresspecialthm}

\newenvironment{remark}[1][]{\refstepcounter{theo} 
\vskip \espaceavantspecialthm \noindent \textsc{Remark~\thetheo
#1.} }%
{\vskip \espaceapresspecialthm}

\def\bb#1{\mathbb{#1}} \def\m#1{\mathcal{#1}}

\def\momeg{(M,\omega)}
\def\co{\colon\thinspace}
\def\Homeo{\mathrm{Homeo}}
\def\Diffeo{\mathrm{Diffeo}}
\def\Symp{\mathrm{Symp}}
\def\Sympeo{\mathrm{Sympeo}}
\def\id{\mathrm{id}}
\newcommand{\norm}[1]{||#1||}
\def\Ham{\mathrm{Ham}}
\def\lagham#1{\mathcal{L}^\mathrm{\Ham}({#1})}
\def\Hamtilde{\widetilde{\mathrm{\Ham}}}
\def\cOlag#1{\mathrm{Sympeo}({#1})}
\def\Crit{\mathrm{Crit}}
\def\diag{\mathrm{diag}}
\def\Spec{\mathrm{Spec}}
\def\osc{\mathrm{osc}}
\def\Cal{\mathrm{Cal}}
\def\Ker{\mathrm{Ker}}
\def\Hom{\mathrm{Hom}}
\def\FS{\mathrm{FS}}
\def\tor{\mathrm{tor}}
\def\Int{\mathrm{Int}}
\def\PD{\mathrm{PD}}
\def\Spec{\mathrm{Spec}}
\def\momeg{(M,\omega)}
\def\co{\colon\thinspace}
\def\Homeo{\mathrm{Homeo}}
\def\Hameo{\mathrm{Hameo}}
\def\Diffeo{\mathrm{Diffeo}}
\def\Symp{\mathrm{Symp}}
\def\Sympeo{\mathrm{Sympeo}}
\def\id{\mathrm{id}}
\def\Im{\mathrm{Im}}
\def\Ham{\mathrm{Ham}}
\def\lagham#1{\mathcal{L}^\mathrm{Ham}({#1})}
\def\Hamtilde{\widetilde{\mathrm{Ham}}}
\def\cOlag#1{\mathrm{Sympeo}({#1})}
\def\Crit{\mathrm{Crit}}
\def\dim{\mathrm{dim}}
\def\Spec{\mathrm{Spec}}
\def\osc{\mathrm{osc}}
\def\Cal{\mathrm{Cal}}
\def\Fix{\mathrm{Fix}}
\def\det{\mathrm{det}}
\def\Ker{\mathrm{Ker}}
\def\coker{\mathrm{coker}}
\def\Per{\mathrm{Per}}
\def\rank{\mathrm{rank}}
\def\Span{\mathrm{Span}}
\def\Supp{\mathrm{Supp}}
\def\Hof{\mathrm{Hof}}
\def\grad{\mathrm{grad}}
\def\ind{\mathrm{ind}}
\def\Hor{\mathrm{Hor}}
\def\Vert{\mathrm{Vert}}
\def\Re{\mathrm{Re}}
\def\van{\mathrm{van}}

\title{Hofer geometry via toric degeneration}
\author{Yusuke Kawamoto}

\newcommand{\Addresses}{{
  \bigskip
  \footnotesize

   \textsc{Yusuke Kawamoto, Institute for Mathematical Research (FIM), R\"amistrasse 101, 8092 Z\"urich Switzerland}\par\nopagebreak
  \textit{E-mail address}: \texttt{yusukekawamoto81@gmail.com, yusuke.kawamoto@math.ethz.ch} }}

\maketitle

\begin{abstract}

The main theme of this paper is to use {\tordeg} to study Hofer geometry by producing distinct {\homo} {\qmor}s on the group of {\hamil} {\diffeo}s. We focus on the (complex $n$-dimensional) quadric hypersurface and the del Pezzo surfaces, and study two classes of distinguished {\lag} sub{\mfd}s that appear naturally in a {\tordeg}, namely the {\lag} torus which is the monotone fiber of a {\lag} torus fibration, and the {\lag} spheres that appear as vanishing cycles. For the quadrics, we prove that the group of {\hamil} {\diffeo}s admits two distinct {\homo} {\qmor}s and derive some superheaviness results. Along the way, we show that the {\tordeg} is compatible with the Biran decomposition. This implies that for $n=2$, the {\lag} fiber torus (Gelfand--Zeitlin torus) is {\hamil} isotopic to the Chekanov torus, which answers a question of Y. Kim. We prove analogous results for the del Pezzo surfaces. We also discuss applications to $C^0$ {\symp} topology.
\end{abstract}

\tableofcontents

\section{Introduction and overview of the results}

\subsection{Hofer geometry}\label{intro hofer}
The main theme of this paper is to use {\tordeg} to study Hofer geometry by producing distinct {\homo} {\qmor}s on the group of {\hamil} {\diffeo}s. The set of {\hamil} {\diffeo}s of a closed {\symp} {\mfd} $X=(X,\omega)$, denoted by $\Ham(X)$ (as well as its universal cover $\wt{\Ham} (X)$) forms a group and moreover, it has a remarkable bi-invariant metric called the Hofer metric \cite{[Hof93]}. The geometry of $\Ham(X)$ (and $\wt{\Ham} (X)$) {\wrt} the Hofer metric has been an active and an important research topic, e.g. \cite{[Pol01]}, which is now called the Hofer geometry. One of the first important questions in this topic was the so-called \textit{Hofer diameter conjecture} which conjectures that the diameter of $\Ham(X)$ (and $\wt{\Ham} (X)$) {\wrt} the Hofer metric is infinity. Many fragmented cases were solved by various {\symp} geometers before a major step forward was made by {\EP} in \cite{[EP03]}, where they introduced the method of \textit{{\qmor}s} to the study of Hofer geometry. This provided a systematic way to solve the Hofer diameter conjecture for a sufficiently large class of {\symp} {\mfd}s by means of Floer theory. However, beyond the Hofer diameter conjecture, not much is known about the Hofer geometry. For example, consider the following question, known as the \textit{Kapovich--{\pol} question}:

\begin{question}[Kapovich--{\pol} question\footnote{The original question was for the case $X=S^2$.}]\label{extended Q of KP}
For a closed {\symp} {\mfd} $X$, is the group $\Ham(X)$ quasi-isometric to the real line $\R$ {\wrt} the Hofer metric?
\end{question}

A negative answer to the Kapovich--{\pol} question will immediately solve the Hofer diameter conjecture, but extremely little is known about it. At the time of writing, the only cases for which Question \ref{extended Q of KP} has been answered (in the negative) are the cases of {\symp} {\mfd}s that satisfy some dynamical condition \cite{Ush13}, $X=S^2 \times S^2$ {\FOOO} \cite{[FOOO19]} (see also \cite{[EliPol]}), and $X=S^2$, which was done in 2021, independently by Cristofaro-Gardiner--Humili\`ere--Seyfaddini \cite{[CGHS]} and {\pol}--Shelukhin \cite{[PS]}. Other than that, even for the weaker version of Question \ref{extended Q of KP} where $\Ham(X)$ is replaced by $\wt{\Ham}(X)$, the question remains widely open.

This paper attempts to use {\tordeg} to obtain new insights about the Hofer geometry, especially in (real) dimension higher than $2$. We focus on the quadric hypersurface
$$Q^n:=\{[z_0:z_1:\cdots :z_{n+1}] \in \C P^{n+1} \colon z_0 ^2 + z_1 ^2 + \cdots+ z_{n+1} ^2 =0  \} $$
and the {\symp} del Pezzo surfaces, and study two distinguished classes of {\lag} sub{\mfd}s that appear naturally in a {\tordeg}, namely the {\lag} torus and the {\lag} spheres that appear as the monotone fiber of a {\lag} torus fibration and vanishing cycles, respectively.

\subsection{Main results for quadrics}\label{main thms quadric}

We first state the results for the quadric hypersurfaces.

\begin{theox}\label{main qmor}
The two {\EP} type {\homo} {\qmor}s 
$$\zeta_{\pm}: \wt{\Ham}(Q^n) \to \R$$
are distinct, i.e. 
$$  \zeta_+ \neq \zeta_- .$$
\end{theox}

This immediately answers the $\wt{\Ham}(X)$ version of Question \ref{extended Q of KP} in the negative for $X=Q^n$. Theorem \ref{main qmor} generalizes results of Eliashberg--Polterovich from \cite{[EliPol]} and the author from \cite{[Kaw22A]}, where the $n=2$-case and the $n=4$-case were proven, respectively.

\subsection{{\lag}s in {\quadr}s}

In the proof of Theorem \ref{main qmor}, producing disjoint {\lag} sub{\mfd}s is a key step and this is where the {\tordeg} plays a crucial role. Recall that roughly speaking, a {\tordeg} for a {\symp} {\mfd} $X$ is a way to deform $X$ into a toric variety $X_0$ which allows one to study the {\symp} geometry of $X$ via the toric geometry of $X_0$ (Section \ref{Toric degeneration}). There are two distinguished subsets for {\tordeg}s. The first is the monotone {\lag} fiber torus; given a {\tordeg}, one obtains a {\lag} torus fibration for $X$ and there is a unique fiber for which the {\lag} fiber torus becomes monotone. The second is the vanishing locus, which is the set obtained by tracing back the points that get crushed into the singular locus of the toric variety $X_0$ in the process of the {\tordeg}. Note that the vanishing locus is not necessarily a geometrically nice set, e.g. a sub{\mfd}, even though in practice it often turns out to be a {\lag} cell-complex.

As for the {\quadr} $Q^n$, the {\tordeg}
\begin{equation}\label{tordeg of our interest}
  \mathfrak{X}:=\{(z,t)   \in \C P^{n+1} \times \C : z_0 ^2 + z_1 ^2 +z_2 ^2 +t(z_3 ^2+ \cdots+ z_{n+1} ^2) =0  \}  
\end{equation}
was introduced by {\NNU} in \cite{[NNU]} which is compatible with the celebrated Gelfand--Zeitlin (GZ) system on $Q^n$. The GZ system defines a {\lag} torus fibration on $Q^n$ and there is one monotone torus fiber which we call the GZ torus and denote it by $T^n _{\text{GZ}}$.

The $n=2$ case deserves special attention, as it is well-known that $Q^2$ is symplectomorphic to the monotone $S^2 \times S^2$. Through this {\symplecto}, we can see the monotone GZ fiber torus $T^2 _{\text{GZ}}$ as a monotone {\lag} torus in $S^2 \times S^2$. In \cite{[KimA]}, Yoosik Kim computed the {\suppot} for the monotone GZ fiber torus $T^n _{\text{GZ}}$ and found that for $n=2$, it agrees with the {\suppot} for the celebrated Chekanov torus in  $S^2 \times S^2$. This lead him to ask the following question:

\begin{question}[{\cite[Section 4.2]{[KimA]}}]\label{Kim's question}
In the monotone $S^2 \times S^2$, is the monotone GZ fiber torus $T^{2} _{\text{GZ}}$ {\hamil} isotopic to the Chekanov torus $T^2 _{\text{Ch}}$?
\end{question}

We study some geometric properties of the distinguished monotone {\lag} sub{\mfd}s $T^n _{\text{GZ}}$ and $S^n $ in $Q^n$. Our second main result is the following.

\begin{theox}\label{lags in Qn}
The monotone {\lag} torus fiber $T^n _{\text{GZ}}$ in $Q^n$ satisfies the following properties:
\begin{enumerate}
    \item Consider the {\polar} $(X,\Sigma)=(Q^n,Q^{n-1})$. The monotone GZ fiber $T^n _{\text{GZ}}$ in $Q^n$ coincides with the distinguished monotone {\lag} torus obtained from the Biran circle bundle construction from $T^{n-1} _{\text{GZ}}$ in $Q^{n-1}$, i.e.
    $$T^n _{\text{GZ}} = \wt{T^{n-1} _{\text{GZ}}} .$$

    \item The real {\lag} sphere
    $$ S_{\van} ^n := \{z  \in \C P^{n+1}: z_0 ^2 + \cdots+z_{n} ^2+ z_{n+1} ^2 =0 ,\ z_0,\cdots,z_n \in \R,\ z_{n+1} \in i \R \} $$
    is a vanishing locus of the {\tordeg} \eqref{tordeg of our interest}, and it is disjoint from $T^n _{\text{GZ}}$, i.e.
    $$ S_{\van} ^n  \cap T^n _{\text{GZ}} = \emptyset .$$

    \item The Gelfand--Zeitlin torus $T^n _{\text{GZ}}$ is $\zeta_+$-{\suphv} and the vanishing cycle $S_{\van} ^n $ is $\zeta_-$-{\suphv}.
\end{enumerate}
\end{theox}

 The third assertion of Theorem \ref{lags in Qn} generalizes a result of Eliashberg--Polterovich from \cite{[EliPol]} where the $n=2$ case was proven.

\begin{remark}
\begin{enumerate}

\item  For details about {\polar}s and the Biran circle bundle construction, see Section \ref{Biran {\decomp}}.

\item The {\sing} toric variety which is the central fiber of the {\tordeg} \eqref{tordeg of our interest} has non-isolated {\sing} locus. This makes it more difficult to understand the geometry of the vanishing locus, c.f. the second statement in {\thm} \ref{lags in Qn}, compared to the cases for del Pezzo surfaces that we consider later, where the {\sing} loci are rational double points ($A_m$-{\sing}ities).

    \item In \cite{[Kaw22A]}, it was proven that $\zeta_+ \neq \zeta_-$ for $n=2,4$ but nothing about the {\suphvness} was proved. The new input is the use of some results from the mirror symmetry literature, e.g. the Auroux--Kontsevich--Seidel theorem, combined with the comparison technique of {\specinv}s with different {\coeff}s developed in \cite{[Kaw22A]}; see Sections \ref{proof-part 1} and \ref{nov and lau} for more on this.

    \item In \cite{[Kaw22A]}, finding disjoint {\lag}s was also a key step, but in this paper, we focus on a different geometric structure and use different {\lag}s; e.g. for $n=4$, \cite{[Kaw22A]} focused on a {\lag} diffeomorphic to $S^1 \times S^3$, which is different from the sphere we consider in {\thm} \ref{lags in Qn}.

    \item Theorem \ref{lags in Qn} has other applications; it plays an important role in \cite{[KawA]}.
\end{enumerate}
\end{remark}

The $n=2$ case of the first assertion in Theorem \ref{lags in Qn} has the following corollary, which answers the aforementioned Question \ref{Kim's question} of Yoosik Kim.

\begin{corol}[Kim's question \ref{Kim's question}]\label{ans Kim's question}
In the monotone $S^2 \times S^2$, the monotone GZ fiber $T^{2} _{\text{GZ}}$ is {\hamil} isotopic to the Chekanov torus $T^2 _{\text{Ch}}$.
\end{corol}

\begin{remark}
The answer to Kim's question \ref{Kim's question}, namely {\cor} \ref{ans Kim's question}, was very recently obtained by Kim himself in \cite{[KimB]}, by a different approach. 
\end{remark}

\subsection{Main results for the del Pezzo surfaces}\label{main thms del pezzo}

In this section, we state the main results for the del Pezzo surfaces, which are analogous to results for quadrics. Recall that the del Pezzo surfaces are smooth (complex) two-dimensional Fano varieties, equipped with a monotone {\symp} form. They are classified and consist of $\C P^2$, $S^2 \times S^2(=Q^2)$, and $\mathbb{D}_k:=\C P^2 \# k \ovl{\C P^2},\ 1 \leq k \leq 8$.

\begin{theox}\label{three qmor del Pezzo}
\begin{enumerate}
    \item For $\mathbb{D}_2$, there exist two distinct {\EP} type {\homo} {\qmor}s 
$$\zeta_{j}: \wt{\Ham}(\mathbb{D}_2) \to \R ,\ j=1,2,$$
$$\zeta_{1} \neq \zeta_{2} .$$

\item For $\mathbb{D}_k,\ k=3,4$, there exist three distinct {\EP} type {\homo} {\qmor}s 
$$\zeta_{j}: \wt{\Ham}(\mathbb{D}_k) \to \R ,\ j=1,2,3,$$
$$\zeta_{i} \neq \zeta_{j} , \text{  if $i\neq j$}.$$
\end{enumerate}

\end{theox}

Theorem \ref{three qmor del Pezzo} leads to the following, which answers
the Kapovich--{\pol} question in the negative.

\begin{theox}[Kapovich--{\pol} question]\label{KP question del Pezzo}
 The group $\Ham(\mathbb{D}_k) $ is not quasi-isometric to the real line $\R$ {\wrt} the Hofer metric when $k=3,4$.
\end{theox}

\begin{remark}
    As mentioned earlier in Section \ref{intro hofer}, to the author's knowledge, the Kapovich--{\pol} question has been only answered (in the negative) for few cases.
\end{remark}

\subsection{Applications}

We discuss some applications of Theorems \ref{main qmor}, \ref{lags in Qn}, \ref{three qmor del Pezzo}. The following question of {\EP}--Py from 2012 has been an important open question, which is referred to as the ``Quasimorphism question'' in the monograph of McDuff--Salamon \cite[Chapter 14, Problem 23]{[MS98]}:

\begin{question}[{\EP}--Py question \cite{[EPP12]}]\label{EPP general}

For a closed symplectic manifold $(X,\omega)$, is there a non-trivial homogeneous quasimorphism on $\Ham(X,\omega)$ which is $C^0$-continuous? If yes, can it be Hofer Lipschitz continuous?
\end{question}

\begin{remark}
\begin{enumerate}
    \item 
The $C^0$-topology on $\Ham (X)$ is induced by the $C^0$-metric of Hamiltonian diffeomorphisms $\phi,\psi \in \Ham(M,\omega)$, which is defined by
$$d_{C^0}(\phi,\psi):=\max_{x\in M}d(\phi(x),\psi(x)),$$
where $d$ denotes the distance on $X$ induced by a fixed Riemannian metric on $X$. Different choices of Riemannian metrics on $X$ will induce equivalent $C^0$-topology. 

\item One point of Question \ref{EPP general} is that currently the relation between the Hofer metric and the $C^0$-metric is not well understood. See, for example, \cite{[JS]} for the latest progress on this topic.

\item The original question was posed for $X=S^2$. This extended version was considered in \cite{[Kaw22A]} before the final resolution of the original question by Cristofaro-Gardiner--Humili\`ere--Mak--Seyfaddini--Smith in \cite{[CGHMSS]}.

\end{enumerate}

\end{remark}

In \cite{[Kaw22A]}, a positive answer to Question \ref{EPP general} was provided for $Q^n,\ n=2,4$. By using Theorems \ref{lags in Qn} and \ref{three qmor del Pezzo}, we manage to generalize this to $Q^n$ for all $n$, and to some del Pezzo surfaces.

\begin{theox}\label{EPP for Qn and del Pezzo}
For $X=Q^n$ and $X=\mathbb{D}_k,\ k=2,3,4$, there exist non-trivial homogeneous quasimorphisms
$$\mu:\Ham(X) \to \mathbb{R}$$
that are both Hofer Lipschitz and $C^0$ continuous.
\end{theox}

\begin{remark}
 \begin{enumerate}
     \item We can actually construct two distinct {\qmor}s with the relevant continuity properties in Theorem \ref{EPP for Qn and del Pezzo} for $\mathbb{D}_k,\ k=3,4$.

     \item It is also known that one can construct {\qmor}s for $\C P^n$ that have continuity properties as in Question \ref{EPP general} by \cite{[KS]}. 
 \end{enumerate}

\end{remark}

The argument of the proof of Theorem \ref{EPP for Qn and del Pezzo} actually shows that the asymptotic spectral norm $\ovl{\gamma}$ 
$$\ovl{\gamma} : \Ham(X)  \to \R$$
$$\ovl{\gamma} (\phi) := \lim_{k \to +\infty} \frac{\gamma(\phi^k)}{k}$$
is $C^0$-continuous for $X=Q^n$ and $X=\mathbb{D}_k,\ k=2,3,4$, even though the $C^0$-continuity of the spectral norm $\gamma$ for none of these {\symp} {\mfd}s is confirmed at the time of writing. Note that the $C^0$-continuity of the spectral norm $\gamma$ is currently proven for the closed surfaces \cite{[Sey13]}, {\symp}ally aspherical {\mfd}s \cite{[BHS21]}, $\C P^n$ \cite{[Sh22]}, and negative monotone {\symp} {\mfd}s \cite{[Kaw22B]}.

\begin{theox}\label{asymp gamma conti}
Let $(X,\omega)$ be a monotone symplectic manifold whose quantum cohomology ring $QH (X;\Lambda)$ is semi-simple. Then, the asymptotic spectral norm
$$\ovl{\gamma} : \Ham(X)  \to \R$$
is $C^0$-continuous. Moreover, for $Q^n$ and $\mathbb{D}_k,\ k=2,3,4$, $\ovl{\gamma}$ is non-trivial, i.e.
$$\ovl{\gamma} \neq 0 .$$ 
\end{theox}

\subsection{Acknowledgements}
I thank Yuhan Sun for a stimulating discussion and for sharing his insight with me. I thank Sobhan Seyfaddini for some useful comments. I greatly appreciate the referee's helpful feedback which improved the exposition of the paper. This work was carried out in several different institutions; \'Ecole Normale Sup\'erieure, Paris while the author was a PhD student, Universit\'e de Montr\'eal while the author was a CRM-postdoctoral fellow at Centre de Recherches Math\'ematiques (CRM) and ETH Z\"urich while the author was a Hermann-Weyl-Instructor at the Institute for Mathematical Research (FIM). The author thanks all institutions for their great hospitality.

\section{Preliminaries}\label{prelim}

We will review results that we will use in the proofs. A special emphasis is put on comparing the conventions in \cite{[OU16]} and other works, as we will use the convention from \cite{[OU16]} in the proofs. We will assume that {\symp} {\mfd}s and {\lag} sub{\mfd}s are all closed and monotone \footnote{Recall that a {\symp} {\mfd} $(X,\omega)$ is monotone if we have $\omega|_{\pi_2(X)}=\kappa \cdot c_1|_{\pi_2(X)}$ for some $\kappa >0$ where $c_1$ is the first Chern class of $TX$ and a {\lag} $L$ in $(X,\omega)$ is monotone if we have $\omega|_{\pi_2(X,L)} = \kappa_L \cdot \mu_L|_{\pi_2(X,L)}$ for some $\kappa_L >0$ where $\mu_L$ is the Maslov class of $L$.}.

\subsection{{\EP} {\qmor}s and (super)heaviness}\label{EP qmors and superheaviness}

First, we review some quantitative Floer theory including {\specinv}s by following \cite{[MS04]}. It is well-known that on a {\symp} {\mfd} $(X,\omega)$, for a non-{\degen} {\hamil} $H:=\{H_t:X \to \R \}_{t\in [0,1]}$ and a choice of a nice {\coeff} field $\Lambda^{\downarrow}$, which will be either the downward Laurent {\coeff}s $\Lambda_{\text{Lau}} ^{\downarrow} $
$$\Lambda_{\text{Lau}} ^{\downarrow} :=\{\sum_{k\leq k_0 } b_k t^k : k_0\in \mathbb{Z},b_k \in \mathbb{C} \}  ,   $$
or the downward Novikov {\coeff}s $\Lambda_{\text{Nov}} ^{\downarrow}$
$$\Lambda_{\text{Nov}} ^{\downarrow}:=\{\sum_{j=1} ^{\infty} a_j T^{\lambda_j} :a_j \in \mathbb{C}, \lambda _j  \in \mathbb{R},\lim_{j\to -\infty} \lambda_j =+\infty \} ,$$
one can construct a filtered Floer homology group $\{HF^\tau(H):=HF^\tau(H;\Lambda^{\downarrow})\}_{\tau \in \R}$ where for two numbers $\tau<\tau'$, the groups $HF^\tau(H;\Lambda^{\downarrow})$ and $HF^{\tau'}(H;\Lambda^{\downarrow})$ are related by a map induced by the inclusion map on the chain level:
$$i_{\tau,\tau'}: HF^\tau(H;\Lambda^{\downarrow}) \longrightarrow HF^{\tau'}(H;\Lambda^{\downarrow}) ,$$
and especially we have
$$i_{\tau}: HF^\tau(H;\Lambda^{\downarrow}) \longrightarrow HF^{}(H;\Lambda^{\downarrow}) ,$$
where $HF^{}(H;\Lambda^{\downarrow})$ is the Floer homology. There is a canonical ring isomorphism called the Piunikhin--Salamon--Schwarz (PSS)-map \cite{[PSS96]}, \cite{[MS04]}
\begin{equation}
    PSS_{H; \Lambda_{} } : QH (X;\Lambda_{}) \to HF(H;\Lambda_{} ^{\downarrow}) , \footnote{Usually, the PSS-isomorphism is a ring isomorphism between the quantum \textit{homology} ring and Floer homology ring $QH_\ast(X;\Lambda_{} ^{\downarrow}) \to HF(H;\Lambda_{} ^{\downarrow})$, but here we compose this with the Poincar\'e duality isomorphism $PD \colon QH(X;\Lambda_{}) \to QH_\ast(X;\Lambda_{} ^{\downarrow}) $ and call the composition the PSS-isomorphism, c.f. \cite[Section 12]{[MS04]}}.
\end{equation}
where $QH(X;\Lambda_{})$ denotes the quantum cohomology ring of $X$ with $\Lambda$-{\coeff}s, i.e.
$$ QH (X;\Lambda_{}) := H^\ast (X;\C) \otimes \Lambda_{} .$$
Here, $\Lambda_{}$ is either 
the Laurent {\coeff}s $\Lambda_{\text{Lau}}$
$$\Lambda_{\text{Lau}} :=\{\sum_{k\geq k_0 } b_k t^k : k_0\in \mathbb{Z},b_k \in \mathbb{C} \}  ,   $$
or the Novikov {\coeff}s $\Lambda_{\text{Nov}}$
$$\Lambda_{\text{Nov}}:=\{\sum_{j=1} ^{\infty} a_j T^{\lambda_j} :a_j \in \mathbb{C}, \lambda _j  \in \mathbb{R},\lim_{j\to +\infty} \lambda_j =+\infty \} $$
chosen accordingly to the set-up of the Floer homology.

\begin{remark}
\begin{enumerate}
    \item 
The choice of the {\coeff} fields will eventually become very important, see Section \ref{nov and lau}. If we do not specify the choice of it and denote it by $\Lambda$, it means that the argument/result holds for both $\Lambda_{\text{Lau}}$ and $\Lambda_{\text{Nov}}$.

\item Nevertheless, it might be helpful to keep in mind that $\Lambda_{\text{Lau}}$ can be embedded to $\Lambda_{\text{Nov}}$ by the inclusion given by $t \mapsto T^{\lambda_0}$, where $\lambda_0$ is a positive generator of $\langle \omega , \pi_2(X) \rangle$.
 
\item The formal variable $t$ in $\Lambda_{\text{Lau}}$ should be seen as representing the sphere with the minimal possible positive area $\lambda_0>0$ (and thus, the minimal possible Maslov index as well, by monotonicity).
 
\end{enumerate}
\end{remark}

\textit{Spectral invariants}, which were introduced by Schwarz \cite{[Sch00]} and developed by Oh \cite{[Oh05]} following the idea of Viterbo \cite{[Vit92]}, are real numbers $\{c (H,a ) \in \R\}$ associated to a pair of a {\hamil} $H$ and a class $a \in QH(X;\Lambda)$ in the following way:
$$c(H,a ) := \inf \{\tau \in \R : PSS_{H; \Lambda_{} } (a) \in \Im (i_{\tau})\} .$$

\begin{remark}
Although the Floer homology is only defined for a non-{\degen} {\hamil} $H$, the {\specinv}s can be defined for any {\hamil} by using the following so-called Hofer continuity property:
$$c(H,a )-c(G,a ) \leq \int_{0} ^1 \left(H_t(x) - G_t(x) \right) dt$$
for any $a \in QH(X),\ H$ and $G$.
\end{remark}

\textit{Triangle inequality.} For {\hamil}s $H$ and $G$, define their catenation $H\#G$ (which is a {\hamil} that generates the path $t \mapsto \phi_H ^t \circ \phi_G ^t$) as
\begin{equation}
    (H\#G)(t,x):=H(t,x)+G(t,(\phi_H ^t)^{-1}(x)) .
\end{equation}
Now, given two classes $a,b \in QH(X;\Lambda)$, we have
\begin{equation}
    c(H,a) + c(G,b) \geq c( H\#G , a \ast b) .
\end{equation}

\textit{Valuation.} For any $a\in  QH(X;\Lambda) \backslash \{0\},$
\begin{equation}
   c (0,a)=\nu(a) 
\end{equation}
where $0$ is the zero-function and $\nu : QH(X;\Lambda) \to \R$ is the natural valuation function
\begin{equation}
    \begin{gathered}
        \nu : QH(X;\Lambda) \to \R \\
        \nu(a):= \nu(\sum_{j=1} ^{\infty} a_j T^{\lambda_j}):=\min\{\lambda_j: a_j \neq 0\}.
    \end{gathered}
\end{equation}

\textit{Spectral norm.} The spectral pseudo-norm for a {\hamil} $H$ is the sum of the {\specinv}s for $H$ and the unit $1_X \in QH(X;\Lambda)$ of the {\QH} and for the inverse {\hamil} $\ovl{H}(t,x):=-H(t, (\phi_H ^t)(x))$, which generates the {\hamil} path $t \mapsto (\phi_H ^t)^{-1}$, and the unit $1_X \in QH(X;\Lambda)$:

\begin{equation}
    \begin{gathered}
        \gamma(H) := c(H,1_X ) +c(\ovl{H},1_X ).
    \end{gathered}
\end{equation}

Now, the spectral norm $ \gamma(\phi)$ for a {\hamil} {\diffeo} $\phi$ is defined as follows.

\begin{equation}
    \begin{gathered}
        \gamma: \Ham (X) \longrightarrow \R\\
        \phi \mapsto \gamma(\phi):=\inf_{\phi_H=\phi} \gamma(H) .
    \end{gathered}
\end{equation}

It is well-known that the spectral norm defines a non-{\degen} metric $d_{\gamma}(\phi,\phi'):=\gamma(\phi^{-1} \circ \phi')$ on $\Ham (X) $.  

\textit{{\EP} {\qmor}s.} {\EP} constructed a special map called the {\qmor} on $\wt{\Ham}(X)$ for under some assumptions via {\specinv}s, which we will briefly review. Recall that a {\qmor} $\mu$ on a group $G$ is a map to the real line $\R$ that satisfies the following two properties:
\begin{enumerate}
    \item There exists a constant $C>0$ such that 
    \begin{equation}\label{constant error}
        |\mu(f \cdot g) -\mu(f)-\mu(g)|<C
    \end{equation}
    for any $f,g \in G$.
    
    \item For any $k \in \Z$ and $f \in G$, we have
    \begin{equation}\label{homogeniety}
        \mu(f^k)=k \cdot \mu(f) .
    \end{equation}
\end{enumerate} 

\begin{theo}[{\cite{[EP03]}}]\label{EP qmor}
Suppose $QH(X;\Lambda)$ has a field factor, i.e. 
$$ QH(X;\Lambda) = Q \oplus A $$
where $Q$ is a field and $A$ is some algebra. Decompose the unit $1_{X}$ of $QH(X;\Lambda)$ {\wrt} this split, i.e. 
$$1_{X}=e + a .$$
Then, the {\asympt} {\specinv} of $\wt{\phi}$ {\wrt} $e$ defines a {\qmor}, i.e.
\begin{equation}
    \begin{gathered}
        \zeta_{e}:\wt{\Ham}(X) \longrightarrow \R \\
         \zeta_{e} ( \wt{\phi}) := \lim_{k \to +\infty} \frac{c (H^{\# k},e )}{k}
    \end{gathered}
\end{equation}
where $H$ is any mean-normalized {\hamil} such that the path $t \mapsto \phi_H ^t $ represents the class $\wt{\phi}$ in $\wt{\Ham}(X)$ and $H^{\# k}$ is the $k$-th catenation of $H$, $ H^k:=\underbrace{H \# H \# \cdots \# H}_{k\text{-times}}$.
\end{theo}

\begin{remark}
When we consider {\EP} {\qmor}s, we take the Laurent {\coeff}s $\Lambda_{\text{Lau}}$ (see Section \ref{nov and lau}) but the condition in Theorem \ref{EP qmor} is equivalent for $\Lambda_{\text{Lau}}$ and $\Lambda_{\text{Nov}}$, see \cite{[EP08]}.
\end{remark}

\begin{example}
As for the quadric {\hypsurf} $Q^n$, the quantum cohomology is semi-simple \cite{Abr00} and with the Laurent {\coeff}s, it splits into a direct sum of two fields (which could be easily seen by the fact that the minimal Chern number $N_{Q^n}=n$)

\begin{equation}
    QH(Q^n;\Lambda_{\text{Lau}} ) = Q_+ \oplus Q_- ,
\end{equation}
where the unit $1_{Q^n}$ splits as
\begin{equation}
\begin{gathered}
     1_{Q^n} = e_+ + e_- ,\\
     e_\pm:=\frac{1_{Q^n} \pm PD([pt]) t}{2} .
\end{gathered}
\end{equation}
Thus, we get two {\EP} {\qmor}s 
\begin{equation}
      \zeta_{\pm}:=\zeta_{e_\pm}:\wt{\Ham}(Q^n) \longrightarrow \R 
\end{equation}
which we prove that they are distinct in Theorem \ref{main qmor}.
\end{example}

\begin{remark}
By slight abuse of notation, we will also see $\zeta_{e}$ as a function on the set of time-independent {\hamil}s:
\begin{equation}
    \begin{gathered}
        \zeta_{e}:C^{\infty}(X) \longrightarrow \R \\
         \zeta_{e} ( H) := \lim_{k \to +\infty} \frac{c (H^{\# k},e )}{k}.
    \end{gathered}
\end{equation}
\end{remark}

\textit{Spectral rigidity.} Not only that {\EP} defined {\qmor}s on $\wt{\Ham}(X)$, they introduced a level of rigidity for subsets in $X$ called (super)heaviness.

\begin{definition}[{\cite{[EP09]},\cite{[EP06]}}]\label{def of heavy}
Take an idempotent $e \in QH(X;\Lambda)$ and denote the {\asympt} {\specinv} {\wrt} $e$ by $\zeta_{e}$. A subset $S$ of $X$ is called
\begin{enumerate}
    \item $e$-heavy (or heavy {\wrt} $e$) if for any time-independent {\hamil} $H:X \to \R$, we have
$$  \inf_{x\in S} H(x)  \leq \zeta_{e} ( H) , $$

\item
$e$-{\suphv} (or {\suphv} {\wrt} $e$) if for any time-independent {\hamil} $H:X \to \R$, we have
$$ \zeta_{e} ( H) \leq \sup_{x\in S} H(x) . $$

\end{enumerate}
\end{definition} 

The following is an easy corollary of the definition of {\suphvness} which is useful.

\begin{prop}[{\cite{[EP09]}}]\label{suphv constant}
Assume the same condition on $QH(X;\Lambda)$ as in Theorem \ref{EP qmor}. Let $S$ be a subset of $X $ that is $\zeta$-{\suphv}. For a time-independent {\hamil} $H:X \to \R$ whose restriction to $S$ is constant, i.e. $H|_{S}\equiv r,\ r\in \R$, we have
$$\zeta(H)=r .$$

\end{prop}

We end this section by giving a criterion for heaviness for {\lag}s, proved by {\FOOO} (there are earlier results with less generality, c.f. \cite{[Alb05]}). Let $L$ be a monotone {\lag} and suppose its Floer cohomology $HF^\ast (L,\rho)$ equipped with a $\C ^\ast$-local system $ \rho: H_1 (L;\Z) \longrightarrow \C^\ast $, c.f. \cite{BC12}, is non-zero. There is a ring homomorphism called the (length 0) closed-open map
$$\mathcal{CO}^0 : QH (X;\Lambda) \to HF^\ast (L,\rho)\footnote{The original notation used in \cite{[FOOO09]} for $\mathcal{CO}^0 $ is $i^\ast _{qm}$.}$$
which is a quantum analogue of the restriction map.

\begin{theo}[{\cite[Theorem 1.6]{[FOOO19]}}]\label{CO map heavy}
If $\mathcal{CO}^0 (e)\neq 0$ for an idempotent $e\in QH (X;\Lambda)$, then $L$ is $\zeta_e$-heavy.
\end{theo}

\begin{remark}\label{suphv semi-simple}
When $\zeta_e$ is {\homo}, e.g. when $e$ is a unit of a field factor of $QH (X;\Lambda)$ and $\zeta_e$ is an {\EP} {\qmor}, then heaviness and {\suphvness} are equivalent so Theorem \ref{CO map heavy} will be good enough to obtain the {\suphvness} of $L$.
\end{remark}

\subsection{Biran decomposition}\label{Biran {\decomp}}

In this section, we briefly review the notion of Biran {\decomp}, which was established by Biran in \cite{[Bir01]}, while comparing the conventions in \cite{[Bir01]} and \cite{[OU16]}. We emphasis that in this paper, the convention of Oakley--Usher \cite{[OU16]} is used.

We start with the Biran's setting \cite[page 412]{[Bir01]}. Let $X$ be a {\symp} K\"ahler {\mfd} $X=(X,\omega_{\text{Bir}})$ with an integral {\symp} form, i.e. $[\omega_{\text{Bir}}] \in H^2(X;\Z) $. Consider a (Donaldson) divisor $\Sigma$ which gives a {\polar} of degree $k$, i.e. the pair $(X,\Sigma)$ satisfies 
\begin{equation}\label{polar Biran}
    PD([\Sigma]) = k [\omega_{\text{Bir}}] \in H^2(X;\Z) .
\end{equation}

    The complement $X\backslash \Sigma$ has the structure of a Stein {\mfd} and thus, one can define the unstable set of a plurisubharmonic function which is called the skeleton $\Delta$. Biran proved that the complement of the skeleton $\Delta$, namely $X\backslash \Delta$, is symplectomorphic to a certain disk bundle (defined with appropriate connection/curvature: we require that the connection $1$-form $\alpha_{\text{Bir}} $ satisfies $\int_{\partial D^2= \{|z|=1\}} \alpha_{\text{Bir}} =1 $):
    \begin{equation}
    \begin{gathered}
        D^2=\{w=r \cdot e^{2\pi \sqrt{-1} \theta} \in \C  \colon r=|w|\leq 1 \} \hookrightarrow D\Sigma \xrightarrow[]{\pi} \Sigma \\
        w \mapsto (w,\zeta) \mapsto \zeta ,
    \end{gathered}
    \end{equation}    
    equipped with the {\symp} form
    \begin{equation}
        \omega_{\text{can};\text{Bir}} := \pi^\ast (k  \omega_{\text{Bir}}|_{\Sigma}) + d(r^2 \alpha_{\text{Bir}} ) ,
    \end{equation}
  where $\omega_{\text{Bir}}|_{\Sigma}:=i^\ast \omega_{\text{Bir}},\ i:\Sigma \hookrightarrow X$, that is,

    \begin{equation}\label{decomp Biran}
        (X\backslash \Delta , \omega_{\text{Bir}}) \xrightarrow[]{\simeq} ( D\Sigma , \frac{1}{k} \omega_{\text{can};\text{Bir}}) .
    \end{equation}

    Note that the area of a fiber disk of the disk bundle $( D\Sigma , \omega_{\text{can};\text{Bir}})$ is 
    $$ \int_{D^2=\{|z|<1 \}} d(r^2 \alpha_{\text{Bir}}) = \int_{\partial D^2} r^2 \alpha_{\text{Bir}} =1 .$$
     Thus, for the Biran {\decomp} with Biran's convention $( D\Sigma , \frac{1}{k} \omega_{\text{can};\text{Bir}})$, the area of the fiber disk is $\frac{1}{k}$, i.e. the radius of the disk is $1/\sqrt{\pi k}$.

    Next, we explain Oakley--Usher's convention in \cite{[OU16]}. The {\symp} form of $X$  $\omega_{\text{OU}}$ is scaled so that
    \begin{equation}\label{form Biran OU}
        \omega_{\text{OU}}=2 \pi \omega_{\mathrm{Bir}} .
    \end{equation}
    Note that unlike $\omega_{\text{Bir}}$, $\omega_{\text{OU}}$ does not represent an integral cohomology class, i.e. $[\omega_{\text{OU}}] \notin H^2(X;\Z)$. Oakley--Usher describes the aforementioned Biran {\decomp} from the relation
    \begin{equation}\label{polar OU}
         \tau \cdot  PD([\Sigma]) =  [\omega_{\text{OU}}] 
    \end{equation}
    for $\tau>0$. 
     Note that from the equations \eqref{polar Biran}, \eqref{polar OU}, and \eqref{form Biran OU}, we obtain
     \begin{equation}\label{tau and k}
         k = \frac{2 \pi }{\tau} .
     \end{equation}
     The Biran {\decomp} {\wrt} Oakley--Usher's convention is as follows:
      \begin{equation}\label{decomp OU}
          (X\backslash \Delta , \omega_{\text{OU}}) \xrightarrow[]{\simeq} ( D\Sigma , \frac{1}{k} \omega_{\text{can};\text{OU}})
      \end{equation}
        
    where
    $$ \omega_{\text{can};\text{OU}} := \pi^\ast (k  \omega_{\text{OU};\Sigma}) + d(r^2 \alpha_{\text{OU}} ) ,$$
    $\omega_{\text{OU}}|_\Sigma:=i^\ast \omega_{\text{OU}}$ where $i\colon \Sigma \hookrightarrow  X$ and $\alpha_{\text{OU}} $ is the connection $1$-form satisfying
    $$ \alpha_{\text{OU}} = 2\pi \alpha_{\text{Bir}} . $$
    
    From the equation \eqref{form Biran OU}, we have
    \begin{equation}\label{can form Biran OU}
        \omega_{\text{can};\text{OU}}=2 \pi \omega_{\text{can};\text{Bir}}
    \end{equation}
    and thus, by using \eqref{tau and k}, we have
   \begin{equation}
      \frac{1}{k} \omega_{\text{can};\text{OU}}  =\frac{2 \pi}{k}  \omega_{\text{can};\text{Bir}} = \tau \omega_{\text{can};\text{Bir}} 
    \end{equation}
    which means that for the Biran {\decomp} with Oakley--Usher's convention $( D\Sigma , \frac{1}{k} \omega_{\text{can};\text{OU}})$, the area of the disk is $\tau$, i.e. the radius of the fiber disk is $\sqrt{\tau/\pi }$.

    We look at the examples we will be focusing on.

    \begin{example}
    For $X=\C P^n$, $\omega_{\text{Bir}}$ is the Fubini--Study form scaled so that 
    $$ \int_{\C P^1} \omega_{\text{Bir}} = 1  . $$
    whereas $\omega_{\text{OU}}$ satisfies
    $$ \int_{\C P^1} \omega_{\text{OU}} = 2 \pi  . $$
    
For $X=Q^n$, which is the main case of our interest, $\omega_{\text{Bir}}$ and $\omega_{\text{OU}}$ are
\begin{equation}
    \begin{gathered}
        \omega_{\text{Bir}} = i^\ast \omega_{\C P^{n+1};\text{Bir}} ,\\
 \omega_{\text{OU}} = i^\ast \omega_{\C P^{n+1};\text{OU}} ,
    \end{gathered}
\end{equation}
where $i: Q^n \hookrightarrow \C P^{n+1}$.
    
    Now, we look at some {\polar}s.
    
    \begin{enumerate}
        \item $(X,\Sigma)= (\C P^n , \C P^{n-1}=\{z_{n+1}=0\}),\ (Q^n,Q^{n-1}=\{z_{n+1}=0\} \cap Q^n)$: As the degree of this {\polar} is $k=1$, we have
        $$ \tau = 2 \pi ,$$
        which implies that the fiber disk has radius $\sqrt{2}$.

        \item $(X,\Sigma)= (\C P^n , Q^{n-1})$: As the degree of this {\polar} is $k=2$, we have
        $$ \tau = 2 \pi /2 = \pi  ,$$
        which implies that the fiber disk has radius $1$.

    \end{enumerate}
    \end{example}

Biran decomposition has been also known as a nice way to construct interesting {\lag}s, which we call the Biran circle bundle construction and explain what it is. Let $L$ be a {\lag} {\smfd} in $\Sigma$. Consider the radius $r>0$ circle bundle associated to the disk bundle $D\Sigma$
$$ D\Sigma_{ |u| = r}:= \{u\in D\Sigma : |u| = r  \} .$$
The set
$$ \wt{L}_r := \pi_{ |u| = r} ^{-1} (L),\ \pi_{ |u| = r}: D\Sigma_{ |u| = r} \to \Sigma $$
defines a {\lag} {\smfd} in $D\Sigma$, which is a circle bundle over $L$. Note that $\pi_{ |u| = r}$ denotes the restricted projection $D\Sigma_{|u| = r} \to \Sigma$. Via the {\symp} identification \eqref{decomp Biran}, we can see $\wt{L}_r$ as a {\lag} {\smfd} in $X \backslash \Sigma $ or $X$.

When $L$ is a monotone in $\Sigma$, then there is a distinguished radius $r_0>0$ for which the lifted {\lag} {\smfd} $\wt{L}$ becomes also monotone in $X$ and according to \cite[{\propo} 6.4.1]{[BC09]}, it satisfies

\begin{equation}
    r_0 ^2 = \frac{2\kappa_L}{2\kappa_L+1}
\end{equation}
where $\kappa_L$ is the monotonicity constant for $L$ in $\Sigma$, i.e. $\omega_{\Sigma}|_{\pi_2(\Sigma,L)} = \kappa_L \cdot \mu_L|_{\pi_2(\Sigma,L)}$. We sometimes call the radius $r_0$ the monotone radius as well.

In the following, the lifted {\lag} {\smfd} $\wt{L}:=\wt{L}_{r_0}$ will always be this distinguished monotone {\lag} {\smfd} in $X$.

\subsection{Various ways to see quadrics}

The aim of this section is to see that the quadric {\hypsurf} $Q^n$ can be identified to some coadjoint orbit $\mathcal{O}_{\lambda}$ of the Lie algebra $\mathfrak{so}(n+1;\R)$.

First of all, as Oakley--Usher it is convenient to see $Q^n$ as the following quotient: first, we have 
    
    \begin{equation}
        \begin{aligned}
          \C P^{n} &=  \C^{n+1}\backslash \{0\} / \C^\ast \\
          & = \{z\in \C^{n+1} : |z|=1\} / S^1 .
        \end{aligned}
    \end{equation}
    Then,
    
    \begin{equation}
        \begin{aligned}
          Q^n &= \{z \in \C P^{n+1}: z_0 ^2 +z_1 ^2+ \cdots +z_{n+1} ^2 = 0 \} \\
          &= \{z \in \C ^{n+2}:|z|=2, z_0 ^2 +z_1 ^2+ \cdots +z_{n+1} ^2 = 0\} / S^1 .
        \end{aligned}
    \end{equation}
    This is the Oakley--Usher way to see $Q^n$. Now, by writing $z=x+iy$, we have
    
     \begin{equation}
        \begin{aligned}
          Q^n &=  \{z \in \C ^{n+2}:|z|=2, z_0 ^2 +z_1 ^2+ \cdots +z_{n+1} ^2 = 0\} / S^1 \\
          & = \{z=x+iy \in \C ^{n+2}:|x|=|y|=1, x \cdot y =0 \} / S^1 .
        \end{aligned}
    \end{equation}
    Thus, a point in $Q^n$ is a orthogonal frame $x,y$ in $\R^{n+2}$ where the rotations are identified, thus this defines a plane in $\R^{n+2}$. Thus,
    $$Q^n = Gr_\R ^+ (2, n+2) ,$$
    where $Gr_\R ^+ (2, n+2)$ denotes the real oriented Grassmanian. Now, the group $SO(n+2;\R)$ acts transitively on $Gr_\R ^+ (2, n+2)$ and the isotropy subgroup of the plane spanned by $(1,0,\cdots,0)$ and $(0,1,0,\cdots,0)$ is $SO(2;\R)\times SO(n;\R)$, thus,
    \begin{equation}\label{group action 1}
      Q^n = Gr_\R (2, n+2) = SO(n+2;\R)/ \left( SO(2;\R)\times SO(n;\R) \right).  
    \end{equation}

Now, $G:=SO(n+2;\R) $ also acts transitively on $\mathfrak{so}(n+2;\R)=\{A: A^t=-A \}$ by the adjoint action. We try to find a element $\lambda \in \mathfrak{so}(n+2;\R) $ so that the stabilizer of this action equals the stabilizer of the previous action, namely $SO(2;\R) \times SO(n;\R)$. In order to do this, we can consider 
\begin{equation}
  \lambda:= 
    \begin{pmatrix}
      0 & \lambda_1 &  0  & \cdots &  0 \\
      - \lambda_1 & 0 & 0  &\cdots & 0  \\
      0 & & \ddots  & &  0 \\
      \vdots & & & & \vdots \\
      0 & & & & 0 
    \end{pmatrix}
    \in \mathfrak{so}(n+2;\R).
\end{equation}
Thus, we now have
\begin{equation}\label{group action 2}
   SO(n+2;\R) / SO(2;\R) \times SO(n;\R)  \xrightarrow[]{\simeq} \mathcal{O}_{\lambda} .  
\end{equation}
 
From \eqref{group action 1} and \eqref{group action 2}, we conclude
\begin{equation}
   Q^n  \xrightarrow[]{\simeq} \mathcal{O}_{\lambda} .  
\end{equation}

To avoid confusion, we will denote this identification by $A$:
\begin{equation}\label{quadric as coadj orbit}
\begin{gathered}
   A: Q^n  \xrightarrow[]{\simeq} \mathcal{O}_{\lambda} ,\\
   z \mapsto A(z) .
\end{gathered}
\end{equation}

The coadjoint orbits are known to possess a canonical {\symp} structure given by the so-called Kirillov--Kostant--Souriau (KKS) form $\omega_{\text{KKS}}$ which satisfies
$$[\omega_{\text{KKS}}]=c_1(T\mathcal{O}_{\lambda})$$
and thus via \eqref{quadric as coadj orbit}, one obtains a KKS form on $Q^n$. The {\symp} form we work with in this paper needs to be appropriately scaled, see Section \ref{convention} for this matter.

\subsection{Gelfand--Zeitlin system for quadrics}\label{Gelfand--Zeitlin system}

In this section, we will review the completely integrable system called the Gelfand--Zeitlin (GZ) system for the quadric $Q^n$ via \eqref{quadric as coadj orbit}, as explained in \cite[Section 2.2]{[KimA]}.

For any $z \in Q^n$, for each $k,\ 2\leq k \leq n+2$, we take the left-upper $k \times k$-submatrix of $A(z)$ which we will denote by $A(z) ^{(k)}$. As $A(z)$ is a skew-symmetric $(n+2) \times (n+2)$ matrix, each $A(z) ^{(k)}$ is also skew-symmetric, and thus its eigenvalues are either all $0$ or $\pm \sqrt{-1} \nu ^{(k)} (z),\ 0\cdots, 0$ where $\nu ^{(k)} (z)>0$. Now, define the following map:
\begin{equation}\label{GZ system for quadric}
    \begin{gathered}
        \Phi: Q^n \longrightarrow \R ^n \\
        z \mapsto (\lambda_1 ^{(2)} (z),\cdots, \lambda_1 ^{(n+1)}(z)),
    \end{gathered}
\end{equation}
where
\begin{equation}
    \lambda_1 ^{(k)}(z):=
    \begin{cases}
    \nu ^{(k+1)} (z) \text{  if either $k \geq 2$ or ($k=1$ and $Pf(A(z) ^{(2)})\geq 0$)} \\
    -\nu ^{(2)} (z) \text{  $k=1$ and $Pf(A(z) ^{(2)})< 0$} ,
    \end{cases}
\end{equation}
where $Pf$ is the Phaffian.

Guillemin--Sternberg proved that $\Phi$ forms a completely integrable system.

\begin{theo}[{\cite{[GS83]}}]
The map $\Phi=(\Phi_1,\cdots, \Phi_n)$ in \eqref{GZ system for quadric} is a completely integrable system on $Q^n$. In fact, the $k$-th component $\Phi_k$ generates a {\hamil } $S^1$-action on $\Phi^{-1}(\{u=(u_1,\cdots,u_n ) \in \R^n: u_k \neq 0\})$.
\end{theo}

In \cite[{\propo} 3.1]{[NNU]}, {\NNU} computed the GZ system $\{\lambda_1 ^{(k)}\}_{2 \leq k \leq n+1}$ for $Q^n$ 
$$\lambda_1 ^{(k)}:Q^n \longrightarrow \R$$
which is as follows:

\begin{equation}
        \begin{aligned}
          \lambda_1 ^{(2)} (z) & = \frac{\lambda}{|z|^2} i(z_1 \ovl{z_2} - \ovl{z_1}z_2),\\
          (\lambda_1 ^{(k)} )^2 (z) & = -   \sum_{1\leq i < j \leq k } \left( \frac{z_i \ovl{z_j} - \ovl{z_i}z_j}{|z|^2}  \right)^2 \\
           & = \left( \frac{\lambda}{|z|^2} \right)^2 ( ( \sum_{j=1} ^k |z_j|^2 )^2 - |\sum_{j=1} ^k z_j ^2| ^2 ) ,\ k \geq 3,\\
           \lambda: & = 2.
        \end{aligned}
    \end{equation}

\begin{remark}
In Kim's paper \cite{[KimA]}, he uses the notation $\{u_k\}_{1 \leq k \leq n}$, where the correspondence is given by 
$$\lambda_1 ^{(k)} = u_{k-1} .$$
\end{remark}

The monotone GZ fiber is expressed as follows \cite[{\propo} 3.7]{[KimA]}:
\begin{equation}\label{GZ torus}
    T^n _{\text{GZ}}:= \Phi_{Q^n} ^{-1}\left( (0,2\cdot \frac{1}{n},2\cdot \frac{ 2}{n},\cdots, 2\cdot\frac{n-1}{n} ) \right) 
\end{equation}
where we used the notation
\begin{equation}
    \begin{gathered}
        \Phi_{Q^n} :Q^n \longrightarrow \R ^n \\
        \Phi_{Q^n} (z):=(\lambda_1 ^{(2)},\lambda_1 ^{(3)},\cdots,\lambda_1 ^{(n+1)}) (z)  .
    \end{gathered}
\end{equation}

\subsection{Convention}\label{convention}
In this section, we will remind some conventions that is used in Section \ref{Proof 1} and \ref{Proof 2}, which are taken from \cite{[OU16]}.

For the {\quadr} $Q^n \subset \C P^{n+1}$, we equip a {\symp} form $\omega$ such that
$$ \omega =  i^\ast \omega_{\C P^{n+1} }$$
where
$$i:Q^n \hookrightarrow \C P^{n+1} $$
is the inclusion and $ \omega_{\C P^{n+1} }$ is the Fubini--Study form scaled so that 
$$ \int_{\C P^1} \omega_{\C P^{n+1} } =2 \pi  .$$

This makes the {\monoconst} to be equal to $\frac{2\pi}{n}$, i.e.
    $$ \omega|_{\pi_2(Q^n)} =   \frac{2\pi}{n} \cdot c_1(TQ^n)|_{\pi_2(Q^n)} . $$
In \cite{[KimA]} and \cite{[NNU]}, they scale the {\symp} form on $Q^n$ so that the {\monoconst} becomes $1$. In fact, they use the Kirillov--Kostant--Souriau (KKS) form $\omega_{\text{KKS}}$ that satisfies 
$$[\omega_{\text{KKS}}]=c_1(TQ^n).$$
These different choices of the normalization cause some rescaling in the results in \cite{[KimA]} and \cite{[NNU]} and in this paper, we use their results in the scaled form which amounts to putting $\lambda=2\pi$ instead of $\lambda =n$ in their results.

\section{Proofs of Theorem \ref{lags in Qn} (1), (2)}\label{Proof 1}

\subsection{Toric degeneration}\label{Toric degeneration}

In this section, as we briefly review {\tordeg}s of {\symp} {\mfd}s and of integrable systems. The latter is a special class of the former which was introduced by {\NNU} in \cite{[NNU10]} that makes the integrable structure on a {\symp} {\mfd} and the toric integrable structure on its degenerated toric variety compatible. We study the particular case of the quadric {\hypsurf} $Q^n$ and we specify the {\tordeg} for the quadric {\hypsurf} $Q^n$ which we will use in this paper. We define two distinguished {\lag}s in $Q^n$, namely the monotone GZ torus $T^n _{\text{GZ}}$ and the {\lag} sphere $S^n _{\text{van}}:=\{x \in \C P^{n+1} : x_{0} ^2+x_{1} ^2+ \cdots +x_{n-1} ^2= x_{n} ^2,\ x_j \in \R \}$. We also discuss its {\suppot}, which was computed by Y. Kim in \cite{[KimA]}.

\begin{definition}[{\cite[Definition 2]{[HK15]},\cite[Section 1]{[Eva]}}]
A {\tordeg} of a {\symp} {\mfd} $X$ is a flat family
$\pi: \mathfrak{X} \to \C$
whose fibers $\{X_t:=\pi^{-1}(t)\}_{t\in \C}$ satisfy the following properties:
\begin{enumerate}
    \item For $t\neq 0$, the fiber $X_t$ is smooth and $X_1$ is isomorphic to $X$.
    \item For $t = 0$, the fiber $X_0$ is a toric variety that is not smooth.
    \item The fibres $X_t$ are projective subvarieties of the same projective space, i.e. there is a morphism $f : \mathfrak{X} \to \C P^N$ such that for every $t\in \C$, $f_t := f|_{X_t} : X_t \to \C P^N$ is an embedding.
\end{enumerate}
\end{definition}

    \begin{remark}\label{all fibers symp}
 In fact, it turns out that all the smooth fibers are symplectomorphic to each other; see \cite[Lemma 1.1]{[Eva]}.

    \end{remark}

We provide an example of a {\tordeg} for the quadric {\hypsurf} $Q^n$ considered by {\NNU} in \cite{[NNU]}, which we will be using throughout the paper.

\begin{example}\label{tordeg for quadrics}
Consider the following:
\begin{equation}
    \mathfrak{X}:=\{(z,t)   \in \C P^{n+1} \times \C : z_0 ^2 + z_1 ^2 +z_2 ^2 +t(z_3 ^2+ \cdots+ z_{n+1} ^2) =0  \} .
\end{equation}
    We define
    \begin{equation}
        \begin{gathered}
            \pi \colon \mathfrak{X} \xrightarrow[]{} \C \\
            (z,t) \mapsto t .
        \end{gathered}
    \end{equation}
 Then for $t =1$, the fiber $X_1$ is isomorphic to $Q^n$ and for $t=0$, the fiber $X_0$ is a {\sing} toric variety (that is isomorphic to the weighted projective space $\C P(1,1,1,2,\cdots,2)$), which is an orbifold whose singular locus $X_0 ^{\text{sing}}$ is 
 \begin{equation}\label{quadric singular locus}
     X_0 ^{\text{sing}}=\{ [0:0:0:z_3:\cdots:z_{n+1}] \in \C P^{n+1} \} .
 \end{equation}
\end{example}

\begin{remark}
    In the {\tordeg} considered in Example \ref{tordeg for quadrics}, the {\sing} locus \eqref{quadric singular locus} is not an isolated set (each {\sing} point is a transverse $A_1$-{\sing}ity). This is very different from the {\tordeg}s for del Pezzo surfaces that we will consider in Section \ref{Proof del Pezzo}, and makes the study of the geometry of vanishing cycles more difficult. 
\end{remark}

We now define a {\tordeg} of a completely integrable system.

\begin{definition}[{\cite[Definition 1.1]{[NNU10]}}]\label{tordeg int system}
Let $X^{2n}=(X,\omega)$ be a {\symp} {\mfd} and $\Phi:X \to \R^n$ a completely integrable system {\wrt} $\omega$. A {\tordeg} of the completely integrable system is a {\tordeg} of $X$
$$\pi: \mathfrak{X} \longrightarrow \C $$
with the following data:
\begin{itemize}
    \item a (piecewise smooth) path $\gamma:[0,1 ] \to \C,\ \gamma(0)=1,\ \gamma(1)=0$,
    
    \item a continuous map
    $$\wt{\Phi}: \mathfrak{X}|_{\gamma([0,1])} \to \R ^n ,$$
    
    \item a flow $\phi_t$ on $\mathfrak{X}|_{\gamma([0,1])}$ defined away from the singular locus of $\mathfrak{X}|_{\gamma([0,1])}$,
\end{itemize}

 satisfy the following properties:
\begin{enumerate}
\item For each $t$, $\Phi_t:=\wt{\Phi}|_{X_t} :X_t \to \R^n$ defines a completely integrable system where for $t=0$, it coincides with the toric system $\Phi_0:X_0 \to \R^n$ and for $t=1$, it coincides with $\Phi:X \to \R^n$.
 
    \item Away from the singular locus, the flow $\phi_t$ restricts to a symplectomorphism between $X_s$ and $X_{s-t}$ that preserves the integrable system: 
    
    \begin{equation}
        \begin{tikzcd}
        (X_s,\omega_s) \arrow[rd,"\Phi_t"] \arrow[rr,"\phi_t"] &  &  (X_{s-t},\omega_{s-t})  \arrow[ld,"\Phi_{s-t}"]\\
        & \R^n & .
        \end{tikzcd}
    \end{equation}
  
\end{enumerate}

\end{definition}

We now would like to construct a {\tordeg} for the GZ system on $Q^n$. The methods of {\NNU} from \cite{[NNU10]} and Harada--Kaveh from \cite{[HK15]} which are based on the \textit{gradient-{\hamil} vector field} (due to Ruan \cite{[Rua01]}) allow us to do that starting from the {\tordeg} of $Q^n$ in Example \ref{tordeg for quadrics}, which is of our main interest.

We briefly review Ruan's idea. Let $\mathfrak{X} $ be an algebraic variety equipped with a K\"ahler form $\wt{\omega}$. Let $\pi:\mathfrak{X} \to \C  $ be a morphism (i.e. an algebraic map) and $\nabla \Re (\pi)$ be the gradient vector field on (the smooth locus of) $\mathfrak{X}$ {\wrt} to the K\"ahler form $\wt{\omega}$ where $\Re (\pi)$ is the real part of the holomorphic function $\pi$. Now, the gradient-{\hamil} vector field $V$ on the smooth locus of $\mathfrak{X}$ is defined by
\begin{equation}
    V:=-\frac{\nabla \Re (\pi)}{|\nabla \Re (\pi)|^2}.
\end{equation}
From the normalization, it follows that 
\begin{equation}
    V (\Re (\pi))  =-1.
\end{equation}

Although the flow of the gradient-{\hamil} vector field $V$ is not complete due to it not being defined on the singular loci, Harada--Kaveh \cite{[HK15]} proved that one can extend the flow continuously on the whole $\mathfrak{X}$ and obtain 
\begin{equation*}
    \phi^t \colon \mathfrak{X} \to \mathfrak{X}
\end{equation*}
which satisfies $\phi^t (X_{t})=X_{0}$.

By using this gradient-{\hamil} flow applied to the {\tordeg} of $Q^n$ in Example \ref{tordeg for quadrics}, {\NNU} obtained a {\tordeg} for the GZ system on $Q^n$. 

\begin{prop}[{\cite[{\propo} 3.1]{[NNU]}, \cite[{\propo} 2.6]{[KimA]}}]\label{GZ system compatible}
The {\tordeg} $\pi: \mathfrak{X} \to \C$ in Example \ref{tordeg for quadrics} together with the gradient-{\hamil} flow $\phi_t$ of $\pi$ defines a {\tordeg} of the GZ system on $Q^n$ (by taking the path $\gamma(t):=1-t$ and $\Phi_t:=\Phi_0 \circ \phi_t$.)
\end{prop}

The compatibility between the GZ system on $X=Q^n$ and the toric system on $X_0 (\simeq \C P(1,1,1,2,\cdots, 2))$ implies that the monotone GZ torus is sent to a toric torus through the {\tordeg} with the gradient-{\hamil} flow, i.e. the monotone GZ torus fiber $T^n _{\text{GZ}}= \Phi^{-1}(x_0)$ satisfies
\begin{equation}\label{GZ torus is a toric fiber}
    \phi_1 (T^n _{\text{GZ}}) = \Phi_0 ^{-1} (x_0) ,
\end{equation}
where $x_0$ is the barycenter of the moment polytope of $\Phi_0$.

    The map $\phi_1$ allows us to define the vanishing loci.

    \begin{definition}
    Let $S$ be a subset of the singular locus of $X_0$, i.e. $S \subset X_0 ^{\text{sing}}$. The vanishing locus corresponding to $S$ is the set $\phi_1 ^{-1}(S)$.
    \end{definition}

Now, we get back to the {\tordeg} \eqref{tordeg for quadrics}, \eqref{GZ system compatible} for $Q^n$ which is of our specific interest.

\begin{prop}\label{vanishing sphere}
The {\lag} sphere 
$$ S^n = \{x \in \C P^{n+1} : x_{0} ^2+x_{1} ^2+ \cdots +x_{n-1} ^2= x_{n} ^2,\ x_j \in \R \} \subset Q^n $$
is the vanishing locus of the set
$$ \{[0:0:0:x_3:\cdots:x_{n}:i x_{n+1}] \in \C P^{n+1} :x_{j} \in \R, \ \forall 3 \leq j \leq n+1  \} \subset X_0 ^{\text{sing}} . $$
\end{prop}

\begin{remark}
Because of {\propo} \ref{vanishing sphere}, we will denote the sphere $S^n$ by $S^n _{\text{van}}$.
\end{remark}

\begin{proof}[Proof of {\propo} \ref{vanishing sphere}]

Consider the following anti-{\symp} involution:
    \begin{equation}
        \begin{gathered}
        \tau : \mathfrak{X} \to \mathfrak{X} \\
        ([z_0 :z_1 :\cdots :z_{n}:z_{n+1}],t) \mapsto ([\ovl{z_0} :\ovl{z_1} :\cdots :\ovl{z_{n}}: - \ovl{z_{n+1}} ],\overline{t}).
        \end{gathered}
    \end{equation}
     For each $t\in [0,1]$, this restricts to an anti-{\symp} involution on $X_t$ :
     \begin{equation}
        \begin{gathered}
        \tau : X_t \to X_t \\
        ([z_0 :z_1 :\cdots :z_{n}:z_{n+1}],t) \mapsto ([\ovl{z_0} :\ovl{z_1} :\cdots :\ovl{z_{n}}: - \ovl{z_{n+1}}],t) .
        \end{gathered}
    \end{equation}
    By using that the K\"ahler metric $g$ (corresponding to the K\"ahler form $\wt{\omega}=\omega_{\FS} \oplus \omega_{\text{stand}}$) on $\mathfrak{X}$  is preserved by the anti-{\symp} involution $\tau$, i.e. $\tau^\ast g=g$, we can show that the gradient-{\hamil} vector field $V$ along $\mathfrak{X}|_{[0,1]}$ is also preserved by the anti-{\symp} involution $\tau$, i.e. $\tau_\ast V (z,t) =V(z,t)$ for $(z,t ) \in \mathfrak{X}|_{[0,1]}$ by 
\begin{equation*}
    \begin{aligned}
        g ( \tau_\ast \nabla \Re (\pi) , W)
        &= (\tau^\ast g) (  \nabla \Re (\pi) , \tau^{-1} _\ast W) \\
        & = g  (  \nabla \Re (\pi) , \tau _\ast W) \\
        & = d \Re (\pi) ( \tau _\ast W) \\
        & = d (\Re (\pi) \circ \tau) ( W) \\
        & = d \Re (\pi)  ( W) \\
        & = g ( \nabla \Re (\pi) , W)
    \end{aligned}
\end{equation*}
for any $W \in T_{(z,t)} \mathfrak{X}$ with $(z,t ) \in \mathfrak{X}|_{[0,1]}$. Note that the fifth line uses that $(z,t ) \in \mathfrak{X}|_{[0,1]}$ (more precisely, that $t \in \R$).
    
    Thus, 
\begin{equation*}
    \phi^t (\Fix(\tau) \cap X_1 )=\Fix(\tau) \cap X_{1-t},
\end{equation*}
where 
\begin{equation}
        \begin{aligned}
           \Fix(\tau) \cap X_{1-t}= \{x \in \C P^{n+1} : x_{0} ^2+x_{1} ^2 +x_{2} ^2 + (1-t)(x_{3} ^2+\cdots +x_{n} ^2+ (ix_{n+1}) ^2)=0 ,\ x_j \in \R \} .
        \end{aligned}
    \end{equation}
The $t=1$ case will give us that 
   \begin{equation*}
    \phi^1 (\Fix(\tau) \cap X_1 )=\Fix(\tau) \cap X_{0},
\end{equation*}
    where
    $$\Fix(\tau) \cap X_1  = \{(x,1) \in \C P^{n+1}\times \C : x_{0} ^2+x_{1} ^2+ \cdots +x_{n} ^2= x_{n+1} ^2 \}  ,$$
    which is equivalent to the {\lag} sphere in $Q^n$ that we are interested in, namely
\begin{equation*}
     S^n = \{x \in \C P^{n+1} : x_{0} ^2+x_{1} ^2+ \cdots +x_{n} ^2= x_{n+1} ^2,\ x_j \in \R \}
\end{equation*}
under the identification between $Q^n$ and $X_1$, and

    $$\Fix(\tau) \cap X_{0} = \{[0:0:0:x_3:\cdots:x_{n}:i x_{n+1}] \in \C P^{n+1} :x_{j} \in \R, \ \forall 3 \leq j \leq n+1  \} \subset X_0 ^{\text{sing}}.$$
    Thus, $S^n$ is the vanishing locus of the singular locus
$$ \{[0:0:0:x_3:\cdots:x_{n}:i x_{n+1}] \in \C P^{n+1} :x_{j} \in \R, \ \forall 3 \leq j \leq n+1  \} . $$
\end{proof}

\begin{remark}
    The argument in the above proof is essentially the same as \cite[Lemma 1.20]{[Eva]}, which is a very useful method to identify vanishing cycles based on \cite[Lemma 1.17]{[Eva]}.
\end{remark}

\begin{remark}
    Note that $S^n$ is just a part of the vanishing locus of the singular locus of $X_0$ and we can detect other parts of the vanishing locus by taking different anti-{\symp} involutions: for example, by looking at the fixed locus of
    \begin{equation}
        \begin{gathered}
        \tau' : X_t \to X_t \\
        ([z_0 :z_1 :\cdots :z_{n}:z_{n+1}],t) \mapsto ([\ovl{z_0} :\ovl{z_1} :\cdots :- \ovl{z_{n}}: - \ovl{z_{n+1}}],t) ,
        \end{gathered}
    \end{equation}
   we can conclude that  
   \begin{equation*}
     \{x \in \C P^{n+1} : x_{0} ^2+x_{1} ^2+ \cdots +x_{n-1} ^2=x_{n} ^2 + x_{n+1} ^2,\ x_j \in \R \} \simeq S^1 \times S^{n-1}
\end{equation*}
is also a part of the vanishing locus.
\end{remark}

Nishinou--Nohara--Ueda applied {\tordeg}s of GZ systems to compute the {\suppot} of the monotone {\lag} torus fiber for some {\tordeg}s in \cite{[NNU10],[NNU]}, which enlarged the cases where one can compute the {\suppot}. This idea inspired {\FOOO} and Y. Kim to compute the {\suppot} of the Chekanov torus in $S^2 \times S^2 \simeq Q^2$ \cite{[FOOO12]} and the GZ torus $T^n _{\text{GZ}}$ in $Q^n$ \cite{[KimA]}, respectively. We now summarize Y. Kim's work from \cite{[KimA]}.

According to Y. Kim, the {\suppot} $ W_{T^n _{\text{GZ}}} : (\C ^\ast)^n \to \C $\footnote{More precisely, the {\suppot} for a {\lag} torus $L \simeq T^n$ is a map $ W_{L} : Hom (H_1(L;\Z),\C^\ast) \to \C $, but it is common to identify the set of local systems $Hom (H_1(L;\Z),\C^\ast)$ with $(\C ^\ast)^n$ by taking a $\Z$-basis of $H_1(L;\Z)$.} of $T^n _{\text{GZ}}$ takes the following form:
\begin{equation}\label{suppot quadric}
    W_{T^n _{\text{GZ}}}(z) = \frac{1}{z_{n}} + \frac{z_{n}}{z_{n-1}} + \cdots + \frac{z_{2}}{z_{1}} + 2z_{2} +z_1 z_2 .
\end{equation}

This has $n$ different (non-{\degen}) critical points $ \rho_j,\ j=0,1,\cdots,n-1$:
\begin{equation}\label{crit point of suppot quadric}
      \rho_j:=(1,\xi_j ^{-(n-1)},\xi_j ^{-(n-2)},\cdots,\xi_j ^{-2},\xi_j ^{-1} ) 
    \end{equation}
where $\xi_j = 4^{1/n} \exp(\frac{2\pi \sqrt{-1} }{n} \cdot j),\ j=0,1,\cdots,n-1$ (each $\xi_j$ satisfies $\xi_j ^n= 4$), and the critical values are as follows:
\begin{equation}\label{crit value of suppot quadric}
        W_{T^n _{\text{GZ}}}( \rho_j  ) = n \cdot \xi_j . 
    \end{equation}
    
This implies that there are $n$ different local systems $\rho_j: H_1(T^n _{\text{GZ}} ;\Z) \to \C^\ast,\ j=1,\cdots,n$ so that the Floer homology of $T^n _{\text{GZ}}$ {\wrt} these local systems are non-zero:
\begin{equation}
    HF(T^n _{\text{GZ}},\rho_j) \neq 0 .
\end{equation}

\begin{remark}
As Y. Kim points out in \cite{[KimA]}, {\FOOO}'s {\suppot} for the Chekanov torus in $S^2 \times S^2 \simeq Q^2$  \cite{[FOOO12]} coincides with the $n=2$ case of his {\suppot} \eqref{suppot quadric} for $T^n _{\text{GZ}}$. This made his ask whether the two dimensional GZ torus $T^2 _{\text{GZ}}$ is {\hamil} isotopic to the Chekanov torus (Question \ref{Kim's question}). We will answer this in the positive in {\cor} \ref{ans Kim's question}.
\end{remark}

\subsection{Proof}
In this section, we prove the main results, namely Theorems \ref{main qmor} and \ref{lags in Qn}, and {\cor} \ref{ans Kim's question}. We first look at how {\cor} \ref{ans Kim's question} follows from Theorem \ref{lags in Qn}.

\begin{proof}[Proof of {\cor} \ref{ans Kim's question}]
First of all, note that $Q^2$ is symplectomorphic to the monotone $S^2 \times S^2$ (with appropriate normalization), i.e.
$$\Phi: Q^2 \xrightarrow[]{\simeq} S^2 \times S^2 ,$$
and through this identification, the {\polar} $(Q^2,Q^1)$ gets translated to the {\polar} $(S^2 \times S^2, \Delta)$ where $\Delta$ denotes the diagonal sphere 
$$\Delta:=\{(x,x) \in S^2 \times S^2 \} ,$$
i.e.
\begin{equation}\label{polar equiv}
    \Phi: (Q^2 , Q^1) \xrightarrow[]{\simeq} ( S^2 \times S^2 , \Delta) .
\end{equation}
Now, our aim is to prove that $\Phi(T^2 _{\text{GZ}} )$ is {\hamil} isotopic to the Chekanov torus. From Theorem \ref{lags in Qn}, we know that $T^2 _{\text{GZ}}$ in $Q^2$ is equal to the Biran circle fibration of $T^1 _{\text{GZ}}$ in $Q^1$, i.e.
$$T^2 _{\text{GZ}} = \wt{T^1 _{\text{GZ}}} . $$
The circle $T^1 _{\text{GZ}}$ in $Q^1$ is the equatorial circle $T^1 _{\text{eq}} $ via the identification between $Q^1$ and $S^2$, i.e.
$$ \Phi(T^1 _{\text{GZ}}) = T^1 _{\text{eq}} .$$
In \cite{[OU16]}, it was proven that for the {\polar} $(S^2 \times S^2, \Delta)$, the {\lag} torus obtained as the Biran circle bundle of the equatorial circle in $ \Delta \simeq S^2$, i.e. $\wt{T^1 _{\text{eq}}} $ (which is denoted by $T_{\text{BC}}$ in the paper), is {\hamil} isotopic to the Chekanov torus (and to the {\FOOO} torus, {\EP} torus, Albers--Frauenfelder torus). From the equivalence of {\polar}s \eqref{polar equiv}, we have
\begin{equation}
    \begin{aligned}
      \Phi(T^2 _{\text{GZ}} ) & = \Phi(\wt{T^1 _{\text{GZ}} } ) \\
      & = \wt{ \Phi (T^1 _{\text{GZ}}) } \\
      & = \wt{  T^1 _{\text{eq}} } = T_{\text{BC}} .
    \end{aligned}
\end{equation}

Thus, we conclude that $\Phi(T^2 _{\text{GZ}} )$ is {\hamil} isotopic to the Chekanov torus in $S^2 \times S^2$.
\end{proof}

We now prove the first assertion of Theorem \ref{lags in Qn}.

\begin{proof}[Proof of Theorem \ref{lags in Qn} (1)]

Recall from Section \ref{Gelfand--Zeitlin system} that the monotone GZ torus $T^n _{\text{GZ}}$ is defined as
\begin{equation}
    T^n _{\text{GZ}}:= \Phi_{Q^n} ^{-1}\left( (0,2\cdot \frac{1}{n},2\cdot \frac{ 2}{n},\cdots, 2\cdot\frac{n-1}{n} ) \right) 
\end{equation}
where $\Phi_{Q^n}$ is 
\begin{equation}
    \begin{gathered}
        \Phi_{Q^n} :Q^n \longrightarrow \R ^n \\
        \Phi_{Q^n} (z):=(\lambda_1 ^{(2)},\lambda_1 ^{(3)},\cdots,\lambda_1 ^{(n+1)}) (z)  ,
    \end{gathered}
\end{equation}
with $\lambda_1 ^{(k)}:Q^n \longrightarrow \R$ such that
\begin{equation}
        \begin{aligned}
          \lambda_1 ^{(2)} (z) & = \frac{\lambda}{|z|^2} i(z_1 \ovl{z_2} - \ovl{z_1}z_2),\\
          (\lambda_1 ^{(k)} )^2 (z) & = -   \sum_{1\leq i < j \leq k } \left( \frac{z_i \ovl{z_j} - \ovl{z_i}z_j}{|z|^2}  \right)^2 \\
           & = \left( \frac{\lambda}{|z|^2} \right)^2 ( ( \sum_{j=1} ^k |z_j|^2 )^2 - |\sum_{j=1} ^k z_j ^2| ^2 ) ,\ k \geq 3,\\
           \lambda: & = 2.
        \end{aligned}
    \end{equation}

From now on, we will be looking at the GZ systems on $Q^n$ and its sub{\mfd} $Q^{n-1}=\{z_{n+1}=0\}$ so we will introduce the notation
$$\lambda_{Q^n} ^{(k)} :Q^n \longrightarrow \R $$
to denote the GZ system for $Q^n$ and 
$$\lambda_{Q^{n-1}} ^{(k)}:Q^{n-1} \longrightarrow \R $$
to denote the GZ system for $Q^{n-1}$.

Now, we prove that the monotone GZ fiber in $Q^n$ is equal to the monotone Biran circle fibration of the monotone GZ fiber in $Q^{n-1}$ for the {\polar}, i.e.
$$T^n _{\text{GZ}} = \wt{T^{n-1} _{\text{GZ}}} .$$
Note that apriori, the two geometric {\constr}s, namely the {\tordeg} and the Biran decomposition has nothing to do with each other.

The Biran {\decomp} associated to the {\polar} $(Q^n,Q^{n-1})$ is expressed in terms of {\coord} as follows \cite[Section 4.2]{[OU16]}:

    \begin{equation}\label{Biran quadric}
        \begin{gathered}
            D Q^{n-1} \xrightarrow[]{\simeq} Q^n \backslash \Delta \\
            (w, \zeta) \mapsto [z_0:\cdots:z_{n+1}] = \left[ \left( 1- \frac{|\zeta|^2}{4} \right)w- \frac{\zeta^2 \ovl{w}}{4} : \left( 1- \frac{|\zeta|^2}{4} \right)^{1/2}\zeta \right] ,
        \end{gathered}
    \end{equation}

     First, we focus on the {\hamil} $\lambda_{Q^n} ^{(n+1)}$. By using 
    $$\sum_{j=0} ^{n+1} z_j ^2 = 0 ,$$
    the expression of $\lambda_1 ^{(n+1)}$ can we rewritten as follows:
    \begin{equation}
        \begin{aligned}
          (\lambda_{Q^n} ^{(n+1)} )^2 
          & =  \left( \frac{\lambda}{|z|^2} \right)^2 ((\sum_{j=0} ^{n} |z_j|^2 )^2 - |\sum_{j=0} ^{n} z_j ^2|^2 ) \\
          & = \left( \frac{\lambda}{|z|^2} \right)^2 ( (|z|^2 - |z_{n+1}|^2)^2 - | - z_{n+1} ^2|^2 ) \\
          & = \left( \frac{\lambda}{|z|^2} \right)^2 (|z|^4 - 2|z|^2|z_{n+1}|^2   ) \\
          & = \lambda^2 (1 - 2 \frac{|z_{n+1}|^2 }{|z|^2}   ) ,
        \end{aligned}
    \end{equation}
    thus, we have
    $$\lambda_{Q^n} ^{(n+1)} (z) = \lambda  (1 - 2 \frac{|z_{n+1}|^2 }{|z|^2}   )^{1/2}  .$$

By using that $|z|^2 = 2$ and that the {\symplecto} \eqref{Biran quadric}, we further have
    \begin{equation}
        \begin{aligned}
          \lambda_{Q^n} ^{(n+1)} (z)  & =  \lambda (1 - 2 \frac{|z_{n+1}|^2 }{|z|^2}   )^{1/2} \\
          & = \lambda (1 - 2 \cdot \frac{1}{2}  \left( 1- \frac{|\zeta|^2}{4} \right)|\zeta|^2  )^{1/2} \\
          & =\lambda (1 - |\zeta|^2 + \frac{|\zeta|^4}{4}   )^{1/2}  \\
          & = \lambda (1 -  \frac{|\zeta|^2}{2}   ) .
        \end{aligned}
    \end{equation}

From \eqref{GZ torus}, the monotone GZ fiber satisfies
   $$\lambda_{Q^n} ^{(n+1)} (z)  =2\cdot \frac{n-1}{n} $$
   which is equivalent to 
   $$\lambda(1 -  \frac{|\zeta|^2}{2})  = 2 \cdot \frac{n-1}{n}   $$
   which is (as $\lambda=2$)
   \begin{equation}\label{last term GZ system}
       \frac{|\zeta|^2}{2}  = \frac{1}{n} .
   \end{equation}
   
   We will see that this is precisely the monotone radius of the Biran circle bundle {\constr} for $(Q^n,Q^{n-1})$. Indeed, by \cite[Section 4.4]{[OU16]} (see also \cite[{\propo} 6.4.1]{[BC09]}), the monotone radius $r_0$ satisfies 
   $$\frac{r_0 ^2}{2} = \frac{\kappa_{Q^n} }{2 } = \frac{2/n }{2 }= \frac{1}{n} $$
   where $\kappa_{Q^n}$ is the {\monoconst} for $Q^n$ (see our convention in Section \ref{convention}). Thus,
    \begin{equation}\label{monotone radius}
        \begin{aligned}
          \frac{r_0 ^2}{2} & = \frac{2\kappa_L }{2 \kappa_L +1} \\
          & = \frac{2 \cdot 1/ 2(n-1)}{2 \cdot 1/ 2(n-1) +1} \\
          & = \frac{1}{n-1} \cdot \frac{n-1}{n} \\
          & = \frac{1}{n} .
        \end{aligned}
    \end{equation}

    Now, we shift our focus to $\lambda_{Q^{n}} ^{(k)}$ where $2 \leq k \leq n$. By using the {\symplecto} \eqref{Biran quadric}, we rewrite the {\hamil}s $\lambda_{Q^{n}} ^{(k)} $ as follows: 
    
    \begin{equation}\label{GZ system Biran}
        \begin{aligned}
          \lambda_{Q^{n}} ^{(k)} (z) & = - \frac{\lambda}{|z|^2} \left( \sum_{1\leq i < j \leq k } (z_i \ovl{z_j} - \ovl{z_i}z_j )^2  \right)^{1/2} \\
           & = -   \frac{\lambda}{|z|^2} \left( \sum_{1\leq i < j \leq k } (2 i \Im (z_i \ovl{z_j}) )^2  \right)^{1/2} \\
          & = - \lambda \left( \sum_{1\leq i < j \leq k } (i \Im (z_i \ovl{z_j}) )^2  \right)^{1/2} \\
           & = - \lambda \left( \sum_{1\leq i < j \leq k } (\left(1- \frac{|\zeta|^2}{2} \right)  i \Im ( w_i\ovl{w_j})  )^2  \right)^{1/2} \\
           & = - \lambda\left(1- \frac{|\zeta|^2}{2} \right) \left( \sum_{1\leq i < j \leq k } (  i \Im ( w_i\ovl{w_j})  )^2  \right)^{1/2} \\
           & =   \left(1- \frac{|\zeta|^2}{2} \right)  \lambda_{Q^{n-1}} ^{(k)} (w) . 
        \end{aligned}
    \end{equation}

By using \eqref{last term GZ system}, we have
\begin{equation}\label{GZ system Biran restricted}
    \begin{aligned}
      \lambda_{Q^{n}} ^{(k)} (z) & =   \left(1- \frac{|\zeta|^2}{2} \right)  \lambda_{Q^{n-1}} ^{(k)} (w) \\
           & =   \left(1- \frac{1}{n} \right)  \lambda_{Q^{n-1}} ^{(k)} (w)\\
           & = \frac{n-1}{n} \cdot \lambda_{Q^{n-1}} ^{(k)} (w) .
    \end{aligned}
\end{equation}

    Finally, we compare the monotone GZ tori $T^{n} _{\text{GZ}}$ and $T^{n-1} _{\text{GZ}}$ by using \eqref{last term GZ system} and \eqref{GZ system Biran restricted}. For $2 \leq k \leq n$, by \eqref{GZ torus} the monotone GZ torus $T^{n} _{\text{GZ}}$ satisfies 
    \begin{equation}
        \lambda_{Q^{n}} ^{(k)} (z) = 2 \cdot \frac{k-2}{n}
    \end{equation}
    which, according to \eqref{GZ system Biran restricted}, is equivalent to 
    \begin{equation}
        \lambda_{Q^{n-1}} ^{(k)} (w) = 2 \cdot \frac{k-2}{n-1} 
    \end{equation}
    which is nothing but the description of $T^{n-1} _{\text{GZ}}$. This implies that $T^n _{\text{GZ}}$ is the monotone Biran circle fibration (i.e. the Biran circle fibration with radius as in \eqref{last term GZ system}) over $T^{n-1} _{\text{GZ}}$, namely
$$T^n _{\text{GZ}} = \wt{T^{n-1} _{\text{GZ}}} .$$
This completes the proof of the first assertion of Theorem \ref{lags in Qn}.

\end{proof}

We now prove the second assertion of Theorem \ref{lags in Qn}. 

\begin{proof}[Proof of Theorem \ref{lags in Qn} (2)]
From {\propo} \ref{vanishing sphere}, we know that the vanishing sphere $S^n _{\text{van}}$ in $X=X_1$ gets mapped to the singular locus of $X_0$ by $\phi_1$, i.e.
\begin{equation}\label{sphere in sing locus}
    \phi_1 (S^n _{\text{van}}) \subset X_0 ^{\text{sing}} .
\end{equation}

On the other hand, we have seen in \eqref{GZ torus is a toric fiber} that from {\propo} \ref{GZ system compatible}, we get 

\begin{equation}
    \phi_1 (T^n _{\text{GZ}}) = \Phi_0^{-1} (x_0) ,
\end{equation}
which implies
\begin{equation}\label{torus in reg locus}
    \phi_1 (T^n _{\text{GZ}}) \subset X_0 \backslash X_0 ^{\text{sing}} .
\end{equation}

The properties \eqref{sphere in sing locus} and \eqref{torus in reg locus} imply
$$ S^n _{\text{van}} \cap T^n _{\text{GZ}} = \emptyset .$$

\end{proof}

\section{Proofs of Theorem \ref{main qmor}, \ref{lags in Qn} (3)}\label{Proof 2}

\subsection{Proof--Part 1}\label{proof-part 1}

\begin{proof}[Proof of Theorem \ref{lags in Qn} (3)]
 We use the following theorem due to Auroux--Kontsevich--Seidel.

\begin{theo}[{\cite[Section 6]{[Aur07]},\cite[Lemma 2.7, {\propo} 2.9]{[Sh16]}}]\label{AKS}
Let $X$ be a closed monotone {\symp} {\mfd} and let $L$ be a monotone {\lag}. Assume that for some local system $\rho$, we have $HF(L,\rho) \neq 0$. The (length 0) closed-open map
$$\mathcal{CO}^0: QH(X;\Lambda_{\text{Nov}} ) \longrightarrow HF(L,\rho)$$
has the following properties:
\begin{enumerate}
    \item For $c_1:=c_1(TX)$, 
$$\mathcal{CO}^0 (c_1)=W_L(\rho) \cdot 1_{(L,\rho)} $$
where $W_L(\rho)$ is the value of the {\suppot} $ W_{L} : Hom (H_1(L;\Z),\C^\ast ) \to \C$ of the {\lag} $L$ equipped with a local system $\rho$.

\item Consider the map 
\begin{equation}\label{c1 multiply}
    c_1\ast - :QH(X;\Lambda_{\text{Nov}} ) \longrightarrow QH(X;\Lambda_{\text{Nov}} )
\end{equation}
and split $QH(X;\Lambda_{\text{Nov}} )$ into generalized eigenspaces {\wrt} $c_1 \ast -$:
\begin{equation}\label{c1 eigen split}
    QH(X;\Lambda_{\text{Nov}} ) = \bigoplus_{w} QH(X;\Lambda_{\text{Nov}} )_w
\end{equation}
where $w$ is an eigenvalue of $c_1 \ast -$. The map
$$\mathcal{CO}^0: QH(X;\Lambda_{\text{Nov}} )_w \longrightarrow HF(L,\rho)$$
is zero if $w \neq W_L (\rho)$ and is a unital homomorphism if $w=W_L(\rho)$.
\end{enumerate}

\end{theo}

\begin{remark}\label{remark on the coeff}
Note that in Theorem \ref{AKS}, it is important that we take the universal Novikov field 
$$\Lambda_{\text{Nov}}:=\{\sum_{j=1} ^{\infty} a_j T^{\lambda_j} :a_j \in \mathbb{C}, \lambda _j  \in \mathbb{R},\lim_{j\to +\infty} \lambda_j =+\infty \} $$
for the quantum cohomology instead of the field of the Laurent series
$$\Lambda_{\text{Lau}} :=\{\sum_{k\geq k_0 } b_k t^k : k_0\in \mathbb{Z},b_k \in \mathbb{C} \}  ,   $$
as $\Lambda_{\text{Nov}}$ is algebraically closed and $\Lambda_{\text{Lau}}$ is not.
\end{remark}

Now, we apply Theorem \ref{AKS} to $Q^n$. The eigenvalues of the map
$$c_1 \ast -: QH(Q^n;\Lambda_{\text{Nov}} )  \longrightarrow QH(Q^n;\Lambda_{\text{Nov}} )  $$
are $0$ and $n\cdot \xi_j,\ j=1,\cdots , n$ where $\{\xi_j\}_{j=1,\cdots , n}$ are solutions to $\xi^n=4$ (see \cite[{\cor} 1.14]{[Sh16]}). Thus, $QH(X;\Lambda_{\text{Nov}} )$ splits into a direct sum of generalized eigenspaces {\wrt} $c_1 \ast -$ as follows:
\begin{equation}\label{c1 eigen split for quadric}
    QH(Q^n;\Lambda_{\text{Nov}} ) = \bigoplus_{1 \leq j \leq n} QH(X;\Lambda_{\text{Nov}} )_{n \cdot \xi_j} \oplus QH(X;\Lambda_{\text{Nov}} )_0 .
\end{equation}

We decompose the unit $1_{Q^n}$ {\wrt} this split:
\begin{equation}
    1_{Q^n} = \sum_{1 \leq j \leq n } e_{n\cdot \xi_j} +e_{0}.
\end{equation}

In view of the {\suppot} computation for $T^n _{\text{GZ}}$ of Y. Kim \eqref{crit point of suppot quadric}, \eqref{crit value of suppot quadric}, Theorem \ref{AKS} implies that 

\begin{equation}\label{CO map torus}
    \begin{gathered}
        \mathcal{CO}^0: QH(X;\Lambda_{\text{Nov}} )_{n \cdot \xi_i} \longrightarrow HF(T^n _{\text{GZ}} , \rho_j) \\
    \mathcal{CO}^0 ( e_{n\cdot \xi_i})= \delta_{i,j} \cdot 1_{(T^n _{\text{GZ}}, \rho_j)}  ,
    \end{gathered}
\end{equation}
where 
\begin{equation}
\delta_{i,j}=
    \begin{cases}
    1, \ \text{if $i=j$} \\
    0,\ \text{if $i\neq j$}.
    \end{cases}
\end{equation}

Similarly, given that $W_{S^n _{\text{van}}}=0$ ($S^n _{\text{van}}$ does not bound any Maslov 2-disks), {\thm} \ref{AKS} implies
\begin{equation}\label{CO map sphere}
    \begin{gathered}
        \mathcal{CO}^0: QH(X;\Lambda_{\text{Nov}} )_0 \longrightarrow HF(S^n _{\text{van}}) \\
    \mathcal{CO}^0 ( e_{0})=1_{S^n _{\text{van}}}   .
    \end{gathered}
\end{equation}

The heaviness criterion Theorem \ref{CO map heavy} applied to \eqref{CO map torus} and \eqref{CO map sphere} imply that $T^n _{\text{GZ}}$ is $\zeta_{e_{n\cdot \xi_j}}$-heavy for all $1 \leq j \leq n$ and $S^n _{\text{van}}$ is $\zeta_{e_0}$-heavy. Moreover, by using that $QH(X;\Lambda_{\text{Nov}} )_{n \cdot \xi_j}$ is a field, we can conclude that $\zeta_{e_{n\cdot \xi_j}}$-{\suphv} for all $1 \leq j \leq n$, see Remark \ref{suphv semi-simple}. On the other hand, as $QH(X;\Lambda_{\text{Nov}} )_0$ is not a field if $n$ is even (see Section \ref{nov and lau} for this in more detail), we will need some extra arguments to show that $S^n _{\text{van}}$ is $\zeta_{e_0}$-{\suphv}, which is the purpose of Sections \ref{nov and lau} and \ref{proof part 2}.

\subsection{The Laurent and Novikov fields}\label{nov and lau}

The aim of this section is to clarify the relation between {\specinv}s defined {\wrt} different {\coeff} fields. Note that $\zeta_\pm: \wt{\Ham}(Q^n) \to \R$ are defined with the Laurent {\coeff}s 
$$\Lambda_{\text{Lau}} =\{\sum_{k\geq k_0 } b_k t^k : k_0\in \mathbb{Z},b_k \in \mathbb{C} \}  ,   $$
while we have worked entirely with the Novikov {\coeff}s 
$$\Lambda_{\text{Nov}}=\{\sum_{j=1} ^{\infty} a_j T^{\lambda_j} :a_j \in \mathbb{C}, \lambda _j  \in \mathbb{R},\lim_{j\to +\infty} \lambda_j =+\infty \} $$
in Section \ref{proof-part 1} (see Remark \ref{remark on the coeff}). Recall that $\Lambda_{\text{Lau}}$ can be embedded to $\Lambda_{\text{Nov}}$ by the inclusion given by $t \mapsto T^{\lambda_0}$, and this inclusion extends to 
$$i: QH(X;\Lambda_{\text{Lau}} ) \hookrightarrow QH(X;\Lambda_{\text{Nov}} ).$$

This subtlety of the choice of the {\coeff} field, which was analyzed in \cite[Section 4.2, 4.5]{[Kaw22A]}, might seem technical but turns out to be very useful and important. To summarize the points from \cite[Section 4.2, 4.5]{[Kaw22A]}, to work with {\specinv}s, e.g. {\EP} {\qmor}s, it is more convenient to work with the Laurent {\coeff} while {\lag} Floer theory is more suited to work with the universal Novikov field, e.g. Theorem \ref{AKS}.

We will focus on the case of $Q^n$. With the Laurent {\coeff}s, the quantum cohomology splits into a direct sum of two fields

\begin{equation}\label{split classic}
    QH(Q^n;\Lambda_{\text{Lau}} ) = Q_+ \oplus Q_- ,
\end{equation}
where the unit $1_{Q^n}$ splits as
\begin{equation}
\begin{gathered}
     1_{Q^n} = e_+ + e_- ,\\
     e_\pm:=\frac{1_{Q^n} \pm PD([pt]) t}{2} .
\end{gathered}
\end{equation}
However, when we consider $QH(Q^n;\Lambda_{\text{Nov}} ) $, $Q_{+}$ and $Q_-$ further splits into finer fields; $Q_{+}$ splits as a direct sum of $n$ fields
$$Q_+ =\bigoplus_{1 \leq i \leq n} Q_{+,i},$$
and 
$Q_{-}$ splits as 
\begin{equation}Q_- = 
    \begin{cases}
    Q_-, \text{ if $n$ is odd}\\
    Q_{-,1} \oplus Q_{-,2}, \text{ if $n$ is even}.
    \end{cases}
\end{equation}
Thus, we have
\begin{equation}\label{split modern}
   QH(Q^n;\Lambda_{\text{Nov}} ) = 
   \begin{cases}
   \left( \bigoplus_{1 \leq i \leq n} Q_{+,i} \right) \oplus Q_-, \text{ if $n$ is odd}\\
        \left( \bigoplus_{1 \leq i \leq n} Q_{+,i} \right) \oplus \left( \bigoplus_{j=1,2} Q_{-,j} \right), \text{ if $n$ is even}
    \end{cases}
\end{equation}
where the unit $1_{Q^n}$ splits as
\begin{equation}
    1_{Q^n} = \begin{cases}
    \sum_{1\leq i \leq n} e_{+,i}+ e_-, \text{ if $n$ is odd}\\
    \sum_{1\leq i \leq n} e_{+,i}+ \sum_{j=1,2} e_{-,j}, \text{ if $n$ is even}.
    \end{cases}
\end{equation}

\begin{equation}
\begin{gathered}
     1_{Q^n} = \sum_{1\leq i \leq n} e_{+,i}+ \sum_{j=1,2} e_{-,j} ,\\
    i (e_+) =\sum_{1\leq i \leq n} e_{+,i},\\
      i (e_-) =\sum_{j=1,2} e_{-,j}.
\end{gathered}
\end{equation}

\begin{remark}
When $n$ is even, the idempotents $e_{-,j},\ j=1,2$ are given by 
\begin{equation}
    e_{-,j}:= \frac{1}{2} \left( 1_{Q^n} \pm PD([S^n])\cdot T^{\lambda_0/2} - PD([pt]) \cdot T^{\lambda_0} \right)
\end{equation}
where $[S^n]$ is the homology class represented by the vanishing sphere $S^n
_{\text{van}}$. The precise expression for $\{e_{+,i}\}$ will be omitted but can be obtained similarly.

\end{remark}

 In Section \ref{proof-part 1}, we have considered yet another split of $QH(Q^n;\Lambda_{\text{Nov}} ) $, namely the eigenvalue decomposition \eqref{c1 eigen split for quadric}:
 \begin{equation}\label{c1 eigen split for quadric 2}
    QH(Q^n;\Lambda_{\text{Nov}} ) = \bigoplus_{1 \leq j \leq n} QH(X;\Lambda_{\text{Nov}} )_{n \cdot \xi_j} \oplus QH(X;\Lambda_{\text{Nov}} )_0 ,
\end{equation}
 and the unit $1_{Q^n}$ is decomposed as
 \begin{equation}
    1_{Q^n} = \sum_{1 \leq j \leq n } e_{n\cdot \xi_j} +e_{0}.
\end{equation}
 
 One can check that $c_1 \ast e_{+,i} =n\cdot \xi_i \cdot  e_{+,i}$ and $c_1 \ast  e_{-,j}=0 $ and thus, the relation between the splits \eqref{split classic}, \eqref{split modern} and \eqref{c1 eigen split for quadric 2} is 
\begin{equation}
    \begin{aligned}
        e_{n\cdot \xi_i} & = e_{+,i} ,\\
        e_+ & = \sum_{1\leq i \leq n} e_{+,i},
    \end{aligned}
\end{equation}
and 
\begin{equation}
    \begin{aligned}
         e_- & = e_0,\\
        e_0 & = \sum_{j=1,2} e_{-,j} , \text{ if $n$ is even}
    \end{aligned}
\end{equation}
which means that $QH(X;\Lambda_{\text{Nov}} )_{n \cdot \xi_j}$ is a field for any $1 \leq j \leq n$ and $QH(X;\Lambda_{\text{Nov}} )_0$ is not (it is a direct sum of two fields). All these are not trivial but easy to see from \cite[Section 7.4]{[Sh16]}.

In \cite[Lemma 31, 32]{[Kaw22A]} (see also \cite[Proof of Theorem 6, Remark 44]{[Kaw22A]}), the author studied the relation between {\specinv}s of a class seen as elements of quantum cohomology with different {\coeff} fields and the lemma implies the following.
  
\begin{theo}\label{comparison lemma}
 We have the following relation between {\asympt} {\specinv}s:
 \begin{equation}
     \begin{gathered}
         \zeta_{e_+} = \zeta_{n\cdot \xi_i}  = \zeta_{e_{+,i}} ,\ i=1,\cdots, n, \\
          \zeta_{e_-} = \zeta_{e_0} = \zeta_{e_{-,j}},\ j=1,2 .
     \end{gathered}
 \end{equation}

\end{theo}

\subsection{Proof--Part 2}\label{proof part 2}

In this section, we will combine results from Sections \ref{proof-part 1} and \ref{nov and lau} to complete the proof of Theorem \ref{lags in Qn}.

In Section \ref{proof-part 1}, we have shown that $T^n _{\text{GZ}}$ is $\zeta_{e_{n\cdot \xi_j}}$-{\suphv} for any $1 \leq j \leq n$ and $S^n _{\text{van}}$ is $\zeta_{e_0}$-heavy and thus, Theorem \ref{comparison lemma} implies that $T^n _{\text{GZ}}$ is $\zeta_{e_{+}}$-{\suphv} and $S^n _{\text{van}}$ is $\zeta_{e_-}$-heavy. As $e_-$ is a unit of a field factor of $QH(X,\Lambda_{\text{Lau}})$, $\zeta_{e_-}$ is {\homo} and thus $S^n _{\text{van}}$ is actually $\zeta_{e_-}$-{\suphv}. We have completed the proof of Theorem \ref{lags in Qn}.
\end{proof}

\subsection{Proof of Theorem \ref{main qmor}}

We prove Theorem \ref{main qmor} by using Theorem \ref{lags in Qn}.

\begin{proof}[Proof of Theorem \ref{main qmor}]

From Theorem \ref{lags in Qn}, the {\lag} sphere $S^n _{\text{van}}$ and the monotone GZ torus $T^n _{\text{GZ}}$ are $e_-$-{\suphv} and $e_+$-{\suphv}, respectively. Also by Theorem \ref{lags in Qn}, $S^n _{\text{van}}$ and $T^n _{\text{GZ}}$ are disjoint so one can take a {\hamil} $H$ on $Q^n$ such that it is time-independent and its restriction to $S^n _{\text{van}}$ and $T^n _{\text{GZ}}$ are $0$ and $1$, respectively, i.e.
$$H|_{S^n _{\text{van}}} \equiv 0,\ H|_{T^n _{\text{GZ}}} \equiv 1 .$$
By {\propo} \ref{suphv constant}, we have
$$ \zeta_- (H) = 0,\ \zeta_+(H)= 1 ,$$
which implies 
$$ \zeta_- \neq \zeta_+ .$$
We have proven Theorem \ref{main qmor}.
\end{proof}

\section{Proof of Theorem \ref{three qmor del Pezzo}}\label{Proof del Pezzo}

In this section, we prove Theorem \ref{three qmor del Pezzo}. In \cite{[Sun20]}, Y. Sun studied {\tordeg}s of del Pezzo surfaces and computed the {\suppot}s for the {\lag} tori that are obtained by {\symp}ally parallel transporting the {\lag} fiber tori of the barycenters of the moment polytopes. Unlike the case of quadrics that we considered earlier, the {\sing} loci for the {\tordeg}s of del Pezzo surfaces are isolated sets, and moreover, consist of $A_m$-{\sing}ities. Thus, we do not need to study the geometry of the vanishing cycles, as they are merely $A_m$-{\config}s of {\lagsph}s.

We now explain how to use Y. Sun's results to obtain Theorem \ref{three qmor del Pezzo}.

\begin{proof}[Proof of Theorem \ref{three qmor del Pezzo}]

We start by summarizing Y. Sun's result. In \cite{[Sun20]}, Y. Sun considered {\tordeg}s for the del Pezzo surfaces to compute the {\suppot}s of the {\lag} tori therein. Y. Sun's {\suppot}s (for the monotone torus fiber without bulk-deformation) have non-{\degen} {\critpt}s and thus, the monotone {\lag} torus fiber has non-zero self-Floer homology {\wrt} the local systems corresponding to the non-{\degen} {\critpt}s of the {\suppot}. Y. Sun also classifies the singularities that appear in each {\tordeg}, and the result is as follows:

\begin{enumerate}
    \item For $\mathbb{D}_2$, there is one $A_1$-{\sing}ity (\cite[Section 3.4]{[Sun20]}).

    \item For $\mathbb{D}_3$, there is one $A_1$-{\sing}ity and one $A_2$-{\sing}ity (\cite[Appendix B, case of $X_3$]{[Sun20]}).

    \item For $\mathbb{D}_4$, there is one $A_1$-{\sing}ity and one $A_2$-{\sing}ity (\cite[Appendix B, case of $X_6$]{[Sun20]}).

\end{enumerate}

\begin{remark}
\begin{enumerate}
        \item Y. Sun considers $\mathbb{D}_5$ as well, but we will not deal with this, as $QH(\mathbb{D}_5)$ is not semi-simple. 

\item For $\mathbb{D}_3$ and $\mathbb{D}_4$, there are several {\tordeg}s; see \cite[Appendix B, cases of $X_3,X_4,X_5$]{[Sun20]}, \cite[Appendix B, case of $X_6, X_7$]{[Sun20]}, respectively.

        \item One difference between the case of the quadrics and the case of del Pezzo surfaces is that in the former, the singular locus was not a discrete set, while in the latter the singular locus is discrete.
    \end{enumerate}
\end{remark}

By looking at the vanishing cycles of the $A_m$-{\sing}ities, one obtains a chain of {\lag} spheres (i.e. an $A_m$-{\config}), and in particular we have the following:

\begin{enumerate}
    \item In $\mathbb{D}_2$, there is one {\lag} sphere $S_1$ that is disjoint from the monotone {\lag} torus.

    \item In $\mathbb{D}_3$, there is one {\lag} sphere $S_1$ (from the $A_1$-{\sing}ity) and one {\lag} sphere $S_2$ (from the $A_2$-{\sing}ity) that are both disjoint to each other and also from the monotone {\lag} torus.

    \item In $\mathbb{D}_4$, there is one {\lag} sphere $S_1$ (from the $A_1$-{\sing}ity) and one {\lag} sphere $S_2$ (from the $A_2$-{\sing}ity) that are both disjoint to each other and also from the monotone {\lag} torus.

\end{enumerate}

\begin{remark}
The properties of the {\specinv}s of the {\lag} spheres forming an $A_m$ {\config} is studied in \cite{[KawB]}. The relation between the spheres and the corresponding {\idem}s is also explained in more detail in \cite{[KawB]}.
\end{remark}

Note that all the {\lag} spheres that appear above have non-zero Floer homology, c.f. \cite[{\propo} A.2]{[BM15]}. Now, using that $QH(\mathbb{D}_k)$ for $k=2,3,4$ are semi-simple, c.f. \cite{[BM04]}, and by employing an analogous argument to Section \ref{Proof 2}, we have the following:
\begin{enumerate}
    \item For $k=2$, there exist two units of field factors of $QH(\mathbb{D}_2)$ $e_1,e_2$ for which the {\lag} sphere $S_1$ is $e_1$-{\suphv} and the monotone {\lag} torus is $e_2$-{\suphv}.

    \item For each $k =3,4$, there exist three units of field factors of $QH(\mathbb{D}_k)$ $e_1,e_2,e_3$ for which the {\lag} sphere $S_1$ is $e_1$-{\suphv}, the {\lag} sphere $S_2$ is $e_2$-{\suphv}, and the monotone {\lag} torus is $e_3$-{\suphv}.
\end{enumerate}

As all the involved {\lag}s are disjoint, we have  
\begin{enumerate}
    \item For $k=2$, $\zeta_{e_1} \neq \zeta_{e_2}$.

    \item For $k =3,4$, $\zeta_{e_i} \neq \zeta_{e_j}$, $1\leq i< j \leq 3$.
\end{enumerate}

This finishes the proof of Theorem \ref{three qmor del Pezzo}.

\end{proof}

\begin{remark}
    For the del Pezzo surfaces, unlike the quadrics, we did not have to deal with the change of {\coeff} as in Section \ref{nov and lau}. This is because the minimal Chern number of the del Pezzo surfaces $\mathbb{D}_k,\ k\geq 1$ is $1$, and thus Laurent {\coeff}s and Novikov {\coeff}s will give the same split in quantum cohomology. 
\end{remark}

\section{Proofs of Theorem \ref{KP question del Pezzo} and Applications}

In this section, we prove Theorems Theorem \ref{KP question del Pezzo}, \ref{EPP for Qn and del Pezzo} and \ref{asymp gamma conti}.

We first prove Theorem \ref{EPP for Qn and del Pezzo}.

\begin{proof}[Proof of Theorem \ref{EPP for Qn and del Pezzo}]
First, recall the following result from \cite{[Kaw22A]}:

\begin{theo}[{\cite[Theorem 22]{[Kaw22A]}}]\label{Kaw22 conti}
Let $(X,\omega)$ be a monotone symplectic manifold. Assume its quantum cohomology ring $QH (X;\Lambda)$ is semi-simple i.e. 
$$QH (X;\Lambda)=Q_1 \oplus Q_2 \oplus \cdots \oplus Q_l$$
for some $l\in \mathbb{N}$ where each $Q_j$ is a field. We decompose the unit $1_X \in QH (X;\Lambda)$ into a sum of idempotents with respect to this split:
$$1_X=e_1+e_2+\cdots+e_l,\ e_j\in Q_j.$$
Then for any $i,j \in \{1,2,\cdots,l\},$
$$\mu:=\zeta_{e_i}-\zeta_{e_j}$$
descends from $\wt{\Ham}(X,\omega)$ to $\Ham(X,\omega)$ and defines a homogeneous quasimorphism on $\Ham(X,\omega)$ which is $C^0$-continuous i.e.
$$\mu :(\Ham(X,\omega),d_{C^0})\to \mathbb{R}$$ is continuous. Moreover, it is Hofer Lipschitz continuous.
\end{theo}

Now, when $X=Q^n$, we take 
$$\mu:= \zeta_{+}-\zeta_{-} .$$
Theorem \ref{Kaw22 conti} and Theorem \ref{main qmor} imply Theorem \ref{EPP for Qn and del Pezzo}.

When $X=\mathbb{D}_k,\ k=2,3,4$, we take 
$$\mu_{i,j}:= \zeta_{e_i}-\zeta_{e_j}$$
with $i \neq j$, where the {\EP} {\qmor}s $\{\zeta_{e_i}\}$ are taken as in Section \ref{Proof del Pezzo}. Theorem \ref{Kaw22 conti} and Theorem \ref{three qmor del Pezzo} imply Theorem \ref{EPP for Qn and del Pezzo}.

\end{proof}

\begin{remark}
Note that the above {\qmor} extends to the $C^0$-closure of $\Ham(X)$, so $\mu=\zeta_+ - \zeta_-$ and $\mu_{i,j}= \zeta_{e_i}-\zeta_{e_j}$ give {\qmor}s on $\ovl{\Ham}^{C^0}(Q^n)$ and $\ovl{\Ham}^{C^0}(\mathbb{D}_k),\ k=2,3,4$, respectively.
\end{remark}

We prove Theorem \ref{KP question del Pezzo}.

\begin{proof}[Proof of Theorem \ref{KP question del Pezzo}]
    When $k=3,4$, in the final step of the proof of Theorem \ref{EPP for Qn and del Pezzo}, we had
\begin{equation}
    \begin{gathered}
        \mu_{1,2} , \mu_{2,3}: \Ham(\mathbb{D}_k ) \to \R \\
        \mu_{1,2} \neq \mu_{2,3} .
    \end{gathered}
\end{equation}
As $\mu_{1,2}$ and $\mu_{2,3}$ are both Hofer Lipschitz continuous and linearly independent, i.e. there is no constant $\alpha \in \R$ such that $\mu_{2,3} (\phi)= \alpha \cdot \mu_{1,2}(\phi) $ for every $\phi \in \Ham (\mathbb{D}_k )$, we can conclude that $\Ham(\mathbb{D}_k )$ is not quasi-isometric to the real line $\R$ {\wrt} the Hofer metric in the following way: first, we see that $\Z$ (and thus $\R$) quasi-isometrically embeds to $\Ham (\mathbb{D}_k )$. Indeed, by taking $\phi \in \Ham (\mathbb{D}_k )$ such that $\mu_{1,2}(\phi) \neq 0$, we can see that
\begin{equation}
\begin{gathered}
       \Z \to \Ham (\mathbb{D}_k ) \\
    k \mapsto \phi^k
\end{gathered}
\end{equation}
defines a quasi-isometric embedding, as 
\begin{equation}
    \begin{aligned}
      |\mu_{1,2} ( \phi)|  \cdot |k-l| & = |\mu_{1,2} ( \phi^{k-l})| \\
      & \leq  d_{\Hof} (\id , \phi^{k-l}) \\
      & =  d_{\Hof} (\phi^k , \phi^l) \\
        & \leq d_{\Hof} (\id , \phi)   \cdot |k-l| .
    \end{aligned}
\end{equation}

Now, we want to show that $\Z$ (thus $\R$) is not quasi-isometric to $\Ham(\mathbb{D}_k )$. Without loss of generality, we can assume that $\mu_{2,3} (\phi)=0$ (otherwise, we can consider $\mu_{2,3} - \frac{\mu_{2,3} (\phi)}{\mu_{1,2} (\phi) } \cdot \mu_{1,2} $ instead of $\mu_{2,3} $). As $\mu_{1,2} $ and $\mu_{2,3} $ are linearly independent, there is $\psi \in \Ham (\mathbb{D}_k )$ such that $\mu_{2,3} (\psi) \neq 0$. Consider the following embedding of $\Z$ to $\Ham (\mathbb{D}_k )$:
\begin{equation}
\begin{gathered}
       \Z \to \Ham (\mathbb{D}_k ) \\
    k \mapsto  \psi^k .
\end{gathered}
\end{equation}

In order to show that $\Z$ (thus $\R$) is not is not quasi-isometric to $\Ham(\mathbb{D}_k )$, it suffices to show that the orbit $\{\psi^k : k \in \Z\}$ gets arbitrary far from $\{\phi^m : m\in \Z \}$. To see this, for every $m \in \Z$, we have
\begin{equation}\label{nuisbicscsjnovc}
    \begin{aligned}
     d_{\Hof} (\psi^k, \phi^m) & =  d_{\Hof} (\id,  \phi^{-m} \psi^k) \\
     & \geq |\mu_{2,3} (\phi^{-m} \psi^k) | \\
     & \geq |\mu_{2,3} ( \psi^k) - \mu_{2,3} (\phi^{m}) | - Const. \\
     & = |k \cdot \mu_{2,3} ( \psi) -m \cdot 0 | - Const. \\
     & = |k| \cdot |\mu_{2,3} ( \psi)  | - Const .
    \end{aligned}
\end{equation}
Note that the third line uses the property \eqref{constant error} of {\qmor}s and the constant only depends on $\mu_{2,3}$ and the fourth line uses the property \eqref{homogeniety} of {\qmor}s. Now, \eqref{nuisbicscsjnovc} shows that 
\begin{equation*}
    \lim_{k \to \pm \infty} d_{\Hof} (\psi^k, \phi^m)=+\infty
\end{equation*}
and thus $\Z$ (thus $\R$) is not is not quasi-isometric to $\Ham(\mathbb{D}_k )$. This completes the proof of Theorem \ref{KP question del Pezzo}.
\end{proof}

We prove {\thm} \ref{asymp gamma conti}.

\begin{proof}[Proof of {\thm} \ref{asymp gamma conti}]
We start with the following claim.

\begin{claim}\label{zeta max}
Suppose $QH(X)$ is semi-simple and denote the split of it by 
$$ QH (X;\Lambda) = Q_1 \oplus Q_2 \oplus \cdots \oplus Q_l.$$
For any {\hamil} $H$, we have
$$ \zeta_{1_X} (H )  = \max_{1 \leq j \leq l}  \zeta_{e_j} (H) .$$
\end{claim}

Claim \ref{zeta max} implies that for any $\phi \in \Ham (X)$,
\begin{equation}\label{useful relation C0}
    \ovl{\gamma} (\phi ) =   \max_{1 \leq i, j \leq l} \mu_{i,j} (\phi)
\end{equation}  
where 
$$ \mu_{i,j} (\phi_H) = \mu_{i,j} (H):=   \zeta_{e_i} (H) - \zeta_{e_j} (H).$$
In fact,  
\begin{equation}
    \begin{aligned}
      \ovl{\gamma} (\phi_H ) & =  \zeta_{1_X} (H ) +  \zeta_{1_X} (\ovl{H} ) \\
      & = \max_{1 \leq i \leq l}  \zeta_{e_i} (H) + \max_{1 \leq j \leq l}  \zeta_{e_j} (\ovl{H}) \\
& = \max_{1 \leq i,j \leq l}  \{ \zeta_{e_i} (H) - \zeta_{e_j} (H) \} \\ & =   \max_{1 \leq i, j \leq l} \mu_{i,j} (\phi) .
    \end{aligned}
\end{equation}

By {\thm} \ref{Kaw22 conti}, we know that for each $i,j$, $\mu_{i,j}$ is $C^0$-continuous and therefore, $ \ovl{\gamma}$ is $C^0$-continuous. {\thm} \ref{EPP for Qn and del Pezzo} implies $\ovl{\gamma} \neq 0$ for $X=Q^n$ and $X=\mathbb{D}_k,\ k=2,3,4$.
\end{proof}

We prove Claim \ref{zeta max}.

\begin{proof}[Proof of Claim \ref{zeta max}]
We first prove $  \zeta_{1_X} (H )  \geq \max_{1 \leq j \leq l}  \zeta_{e_j} (H) $ for any $H$. By the triangle inequality, we get
$$ c(H , 1_X) + \nu (e_j) \geq c (H , e_j) $$
for any $j$ and {\hamil} $H$ and thus
$$ \zeta_{1_X} (H) \geq   \zeta_{e_j} (H) $$ 
for any $j$ and {\hamil} $H$. Therefore, 
$$  \zeta_{1_X} (H )  \geq \max_{1 \leq j \leq l}  \zeta_{e_j} (H) $$
for any {\hamil} $H$. 

Next, we prove $  \zeta_{1_X} (H )  \leq \max_{1 \leq j \leq l}  \zeta_{e_j} (H)$. A standard property of spectral invariants implies 
$$ c (H , 1_X)  \leq \max_{1 \leq j \leq l}  c (H ,e_j) $$
as $1_X= e_1 +e_2 + \cdots + e_l$ and thus
$$ \zeta_{1_X} (H )  \leq \max_{1 \leq j \leq l}  \zeta_{e_j} (H) $$
for any {\hamil} $H$. This completes the proof of Claim \ref{zeta max}.

\end{proof}

\appendix

\Addresses


\begin{thebibliography}{99}




\bibitem[Abr00]{Abr00} Lowell Abrams, The quantum Euler class and the quantum cohomology of the Grassmannians, \emph{Isr. J. Math.} 117, 335-352 (2000).

\bibitem[Alb05]{[Alb05]} Peter Albers, \emph{On the extrinsic topology of Lagrangian submanifolds}, 
IMRN 2005, 38, 2341--2371, Erratum IMRN 2010 7, 1363--1369



\bibitem[Aur07]{[Aur07]} Denis Auroux, Mirror symmetry and T-duality in the complement of an anticanonical divisor. \emph{J. G\"okova Geom. Topol.}, GGT 1 (2007), 51–-91.


\bibitem[Bir01]{[Bir01]} Paul Biran, \emph{Lagrangian barriers and symplectic embeddings}, Geom. Funct. Anal. 11 (2001), no. 3, 407--464. 

\bibitem[Bir06]{[Bir06]} Paul Biran, \emph{Lagrangian non-intersections}, Geom. Funct. Anal. 16 (2006), no. 2, 279--326. 



\bibitem[BC09]{[BC09]} Paul Biran, Octav Cornea, \emph{Rigidity and uniruling for Lagrangian submanifolds}, Geom. Topol. 13 (2009), no. 5, 2881--2989. 


\bibitem[BHS21]{[BHS21]} Lev Buhovsky, Vincent Humili\`ere, Sobhan Seyfaddini, The action spectrum and $C^0$ symplectic topology, \emph{Math. Ann. } 380 (2021), no. 1-2, 293--316. 


\bibitem[BM04]{[BM04]} Arend Bayer, Yuri I. Manin, (Semi)simple exercises in quantum cohomology, The Fano Conference, 143--173, Univ. Torino, Turin, 2004. 



\bibitem[BM15]{[BM15]} Paul Biran, Cedric Membrez, The Lagrangian cubic equation, \emph{Int. Math. Res. Not. IMRN }
2016, no. 9, 2569--2631.





\bibitem[CGHS]{[CGHS]} Daniel Cristofaro-Gardiner, Vincent Humili\`ere, Sobhan Seyfaddini, PFH spectral invariants on the two-sphere and the large scale geometry of Hofer's metric, arXiv:2102.04404v3, To appear in \emph{J. Eur. Math. Soc. (JEMS).}



\bibitem[CGHMSS]{[CGHMSS]} Daniel Cristofaro-Gardiner, Vincent Humili\`ere, Cheuk Yu Mak, Sobhan Seyfaddini, Ivan Smith, \emph{Quantitative Heegaard Floer cohomology and the Calabi invariant},  arXiv:2105.11026v1





\bibitem[BC09]{[BC09]} Paul Biran, Octav Cornea, \emph{Rigidity and uniruling for Lagrangian submanifolds}, Geom. Topol. 13 (2009), no. 5, 2881--2989. 




\bibitem[BC12]{BC12} Paul Biran, Octav Cornea, Lagrangian topology and enumerative geometry, \emph{Geom. Topol.} 16, No. 2, 963-1052 (2012).


\bibitem[EliPol]{[EliPol]} Yakov Eliashberg, Leonid Polterovich, Symplectic quasi-states on the quadric surface and Lagrangian submanifolds, arXiv:1006.2501v1



\bibitem[EP03]{[EP03]} Michael Entov, Leonid Polterovich, \emph{Calabi quasimorphism and quantum homology},  Int. Math. Res. Not. 2003, no. 30, 1635--1676. 



\bibitem[EP06]{[EP06]} Michael Entov, Leonid Polterovich, \emph{Quasi-states and symplectic intersections}, Comment. Math. Helv. 81 (2006), 75--99




\bibitem[EP08]{[EP08]} Michael Entov, Leonid Polterovich, \emph{Symplectic quasi-states and semi-simplicity of quantum homology}, Toric Topology (eds. M.Harada, Y.Karshon, M.Masuda and T.Panov), 47--70, Contemporary Mathematics 460, AMS, 2008.


\bibitem[EP09]{[EP09]} Michael Entov, Leonid Polterovich, \emph{Rigid subsets of symplectic manifolds}, Compos. Math. 145 (2009), no. 3, 773--826.





\bibitem[EPP12]{[EPP12]} Michael Entov, Leonid Polterovich, Pierre Py, \emph{On continuity of quasimorphisms for symplectic maps}, With an appendix by Michael Khanevsky. Progr. Math., 296, Perspectives in analysis, geometry, and topology, 169--197, Birkh\"auser/Springer, New York, 2012.





\bibitem[Eva]{[Eva]} Jonathan Evans, KIAS Lectures on Symplectic aspects of degenerations,  arXiv:2403.03519v1


\bibitem[FOOO09]{[FOOO09]} Kenji Fukaya, Yong-Geun Oh, Hiroshi Ohta, Kaoru Ono, \emph{Lagrangian intersection Floer theory: anomaly and obstruction. Part I.}, American Mathematical Society, Providence, RI; International Press, Somerville, MA, 2009



\bibitem[FOOO12]{[FOOO12]} Kenji Fukaya, Yong-Geun Oh, Hiroshi Ohta, Kaoru Ono, \emph{Toric degeneration and nondisplaceable Lagrangian tori in $S^2\times S^2$}, Int. Math. Res. Not. 2012, no.13, 2942--2993. 








\bibitem[FOOO19]{[FOOO19]} Kenji Fukaya, Yong-Geun Oh, Hiroshi Ohta, Kaoru Ono, \emph{Spectral invariants with bulk, quasi-morphisms and Lagrangian Floer theory}, Mem. Amer. Math. Soc. 260 (2019), no.1254





\bibitem[GS83]{[GS83]} Victor Guillemin, Shlomo Sternberg, \emph{The Gelfand--Cetlin system and quantization of the complex flag manifolds}, J. Funct. Anal. 52 (1983), no. 1, 106--128. 



\bibitem[HK15]{[HK15]} Megumi Harada, Kiumars Kaveh, Integrable systems, toric degenerations and Okounkov bodies, \emph{Invent. Math.}, 202 (2015), no. 3, 927–985. 


\bibitem[Hof93]{[Hof93]} Helmut Hofer, Estimates for the energy of a symplectic map, \emph{Comment. Math. Helv.} 68 (1993), no. 1, 48–-72. 


\bibitem[JS]{[JS]} Du\v{s}an Joksimovi\'c, Sobhan Seyfaddini, A H\"older-type inequality for the $C^0$ distance and Anosov-Katok pseudo-rotations, \emph{preprint}, arXiv:2207.11813v1   






\bibitem[Kaw22A]{[Kaw22A]} Yusuke Kawamoto, Homogeneous quasimorphisms, $C^0$-topology and Lagrangian intersection, \emph{Comment. Math. Helv.} 97 (2022), no. 2, pp. 209–254


\bibitem[Kaw22B]{[Kaw22B]} Yusuke Kawamoto, On $C^0$-Continuity of the Spectral Norm for Symplectically Non-Aspherical Manifolds, \emph{Int. Math. Res. Not. IMRN} 2022, no. 21, 17187--17230.



\bibitem[KawA]{[KawA]} Yusuke Kawamoto, Donaldson divisors and {\EP} {\qmor}s, \textit{to appear}



\bibitem[KawB]{[KawB]} Yusuke Kawamoto, Isolated hypersurface singularities, spectral invariants, and quantum cohomology, arXiv:2304.01847v1, https://doi.org/10.1515/crelle-2024-0013, to appear in \emph{Journal f\"ur die reine und angewandte Mathematik (Crelle's Journal)}, 2024.





\bibitem[KS]{[KS]} Yusuke Kawamoto, Egor Shelukhin, Spectral invariants over the integers, arXiv:2310.19033v1





\bibitem[KimA]{[KimA]} Yoosik Kim, Disk potential functions for quadrics, \emph{J. Fixed Point Theory Appl.} 25, No. 2, Paper No. 46, 31 p. (2023).


\bibitem[KimB]{[KimB]} Yoosik Kim, Chekanov torus and Gelfand--Zeitlin torus in $S^2 \times S^2$, \emph{preprint}, arXiv:2109.01435v1



\bibitem[MS98]{[MS98]} Dusa McDuff, Dietmar Salamon, \emph{Introduction to symplectic topology}, Third edition. Oxford Mathematical Monographs. The Clarendon Press, Oxford University Press



\bibitem[MS04]{[MS04]} Dusa McDuff, Dietmar Salamon, \emph{J-holomorphic Curves and Symplectic Topology: Second Edition}, American Mathematical Society Colloquium Publications, 52. American Mathematical Society, Providence, RI, 2004 
 
 

\bibitem[NNU]{[NNU]} Yuichi Nohara, Takeo Nishinou, Kazushi Ueda, Potential functions via toric degenerations, arXiv:0812.0066v2
 
 
\bibitem[NU16]{[NU16]} Yuichi Nohara, Kazushi Ueda, \emph{Floer cohomologies of non-torus fibers of the Gelfand--Cetlin system}, J. Symplectic Geom. 14 (2016), no. 4, 1251--1293. 


\bibitem[NNU10]{[NNU10]} Yuichi Nohara, Takeo Nishinou, Kazushi Ueda, \emph{Toric degenerations of Gelfand--Cetlin systems and potential functions}, Adv. Math. 224 (2010), no. 2, 648--706. 



\bibitem[Oh05]{[Oh05]} Yong-Geun Oh, \emph{Construction of spectral invariants of Hamiltonian paths on closed symplectic manifolds}, The breadth of symplectic and Poisson geometry, 525--570, Progr. Math., 232 2005


 
 
\bibitem[OU16]{[OU16]} Joel Oakley, Michael Usher, \emph{On certain Lagrangian submanifolds of $S^2\times S^2$ and $\mathbb{C}P^n$}, Algebr. Geom. Topol. Volume 16, Number 1 (2016), 149--209.




\bibitem[PSS96]{[PSS96]} Sergey Piunikhin, Dietmar Salamon, Matthias Schwarz, Symplectic Floer–Donaldson theory and quantum cohomology. Contact and Symplectic Geometry. Cambridge University Press. (1996) pp. 171–200.



\bibitem[Pol01]{[Pol01]} Leonid Polterovich, The geometry of the group of symplectic diffeomorphisms, \emph{Lectures in Mathematics ETH Z\"urich. Birkh\"auser Verlag}, Basel, 2001.



\bibitem[PS]{[PS]} Leonid Polterovich, Egor Shelukhin, Lagrangian configurations and Hamiltonian maps, arXiv:2102.06118v3




\bibitem[Rua01]{[Rua01]} Wei-Dong Ruan, Lagrangian torus fibration of quintic hypersurfaces. I. Fermat quintic case, Winter School on Mirror Symmetry, Vector Bundles and Lagrangian Submanifolds (Cambridge, MA, 1999), 297–332, AMS/IP Stud. Adv. Math., 23, Amer. Math. Soc., Providence, RI, 2001. 
 
 
 
\bibitem[Sch00]{[Sch00]} Matthias Schwarz, \emph{On the action spectrum for closed symplectically aspherical manifolds}, Pacific J. Math. 193 (2000), no. 2, 419--461.




\bibitem[Sey13]{[Sey13]}  Sobhan Seyfaddini, $C^0$-limits of Hamiltonian paths and the Oh-Schwarz spectral invariants, \emph{Int. Math. Res. Not. IMRN} 2013, no. 21, 4920--4960.




\bibitem[Sh22]{[Sh22]} Egor Shelukhin, Viterbo conjecture for Zoll symmetric spaces, \emph{ Invent. Math.} 230 (2022), no. 1, 321--373



\bibitem[Sh16]{[Sh16]} Nick Sheridan, On the Fukaya category of a Fano hypersurface in projective space, \emph{
Publ. Math. Inst. Hautes \'Etudes Sci.} 124 (2016), 165--317. 




\bibitem[Sun20]{[Sun20]} Yuhan Sun, $A_n$-type surface singularity and nondisplaceable Lagrangian tori, \emph{Internat. J. Math.} 31 (2020), no. 3



\bibitem[Ush13]{Ush13} Michael Usher, Hofer's metrics and boundary depth, \emph{Ann. Sci. Éc. Norm. Supér.} (4)46(2013), no.1, 57–128.

\bibitem[Vit92]{[Vit92]} Claude Viterbo, \emph{Symplectic topology as the geometry of generating functions}, Math. Ann. 292 (1992), no. 4, 685--710. 













\end{thebibliography}
\end{document}